\documentclass[a4paper,11pt]{scrartcl}
\pdfoutput=1
\title{Iteratively regularized Newton-type methods for general data misfit functionals and applications to Poisson data}
\author{Thorsten Hohage and Frank Werner \\[0.2cm]\small hohage@math.uni-goettingen.de, +49 (0)551 39 4509\\\small f.werner@math.uni-goettingen.de, +49 (0)551 39 12468\\[0.5cm]
\small Institute for Numerical and Applied Mathematics, University of G\"ottingen\\
\small Lotzestra\ss e 16-18\\
\small 37083 G\"ottingen}
\date{\today}
\usepackage[english] {babel}
\usepackage[utf8]{inputenc}
\usepackage{amsfonts} 
\usepackage{amssymb}
\usepackage{amsthm}
\usepackage{amsxtra}
\usepackage{mathrsfs}
\usepackage{exscale}
\usepackage{graphicx}
\usepackage{subfigure}
\usepackage{cite}
\usepackage{pgfplots}
\usepackage{tikz}
\newlength\fheight \newlength\fwidth
\DeclareMathOperator*{\argmin}{argmin}
\DeclareMathOperator{\err}{\mathbf{err}}
\newcommand{\KL}[2]{\mathbb{KL}\left(#2;#1\right)}

\newcommand{\EW}{\mathbf{E}}
\newcommand{\prob}{\mathbf{P}}
\newcommand{\Var}{\mathbf{Var}} 
\newcommand{\Xspace}{\mathcal{X}}
\newcommand{\Yspace}{\mathcal{Y}}
\newcommand{\Zspace}{\mathcal{Y}^{\mathrm{obs}}}
\newcommand{\gdelta}{g^{\mathrm{obs}}}
\newcommand{\calS}{\mathcal S}
\newcommand{\calT}{\mathcal T}
\renewcommand{\S}[2]{\calS\left(#2;#1\right)}
\newcommand{\Stil}[2]{\tilde{\calS}\left(#2;#1\right)}
\newcommand{\T}[2]{\calT\left(#2;#1\right)}
\newcommand{\offset}{\sigma}
\newcommand{\manifold}{\mathbb{M}}

\newcommand{\Cerr}{C_{\rm err}}
\newcommand{\Ctc}{C_{\rm tc}}
\newcommand{\Cdec}{C_{\rm dec}}
\newcommand{\Cbreg}{C_{\rm bd}}

\newcommand{\Nset}{\mathbb{N}}
\newcommand{\paren}[1]{\left( #1 \right)}

\newtheoremstyle{ass}{.1cm}{.1cm}{\normalfont}{}{\normalfont\bfseries}{:}{ }{\thmname{#1}\thmnumber{ #2}\thmnote{ (#3)}}
\newtheoremstyle{assa}{.1cm}{.1cm}{\normalfont}{}{\normalfont\bfseries}{:}{ }{\thmname{#1}\thmnumber{ #2A}\thmnote{ (#3)}}
\newtheoremstyle{assb}{.1cm}{.1cm}{\normalfont}{}{\normalfont\bfseries}{:}{ }{\thmname{#1}\thmnumber{ #2B}\thmnote{ (#3)}}
\newtheoremstyle{assp}{.1cm}{.1cm}{\normalfont}{}{\normalfont\bfseries}{:}{ }{\thmname{#1}\thmnote{ (#3)}}
\begin{document}
\newtheorem{thm}{Theorem}[section]
\newtheorem{lem}[thm]{Lemma}
\newtheorem{cor}[thm]{Corollary}
\newtheorem{rem}[thm]{Remark}
\newtheorem{example}[thm]{Example}
\renewcommand{\qedsymbol}{\rule{.2cm}{.2cm}}
\theoremstyle{ass}
\newtheorem{ass}{Assumption}
\theoremstyle{assa}
\newtheorem{assa}[ass]{Assumption}
\theoremstyle{assb}
\newtheorem{assb}[ass]{Assumption}
\theoremstyle{assp}
\newtheorem{assp}[ass]{Assumption $\mathcal P$}
\setlength{\parindent}{0cm}
\maketitle

\begin{abstract}
We study Newton type methods for inverse problems described by nonlinear operator equations $F(u)=g$ 
in Banach spaces where the Newton equations $F'(u_n;u_{n+1}-u_n) = g-F(u_n)$ are regularized variationally 
using a general data misfit functional and a convex regularization term. This generalizes the
well-known iteratively regularized Gauss-Newton method (IRGNM). We prove convergence and convergence rates
as the noise level tends to $0$ both for an a priori stopping rule and for a Lepski{\u\i}-type  
a posteriori stopping rule. 
Our analysis includes previous order optimal convergence rate results
for the IRGNM as special cases. The main focus of this paper 
is on inverse problems with Poisson data where the natural data misfit functional 
is given by the Kullback-Leibler divergence. Two examples of such problems are discussed in detail:
an inverse obstacle scattering problem with amplitude data of the far-field pattern 
and a phase retrieval problem. The performence of the proposed method for these problems
is illustrated in numerical examples. 
\end{abstract}
\pagestyle{plain}

\section{Introduction}\label{sec:intro}
This study has been motivated by applications in photonic imaging, e.g.\ 
po\-si\-tron emission tomography \cite{ksv85}, deconvolution problems in 
astronomy and microscopy \cite{bbdv09}, phase retrieval problems \cite{h89} or 
semi-blind deconvolution problems, i.e.\ deconvolution with partially unknown 
convolution kernel \cite{SBH:12}.  
In these problems, data consist of counts of photons which have interacted with the object
of interest. The inverse problem of recovering the information on the object of interest from such photon counts can be formulated as an operator equation
\begin{equation}\label{eq:opeq} 
F \left(u\right) = g
\end{equation}
if one introduces an operator $F:\mathfrak{B}\subset \Xspace\to\Yspace$ 
mapping a mathematical description $u\in\mathfrak{B}$ of the
object of interest to the photon density $g\in \Yspace\subset L^1(\manifold)$
on the manifold $\manifold$ at which measurements
are taken. In this paper we focus on problems where the operator $F$ is 
nonlinear. 

For fundamental physical reasons, photon count data are described by a Poisson
process with the exact data $g^{\dagger}$ as mean if read-out noise and finite averaging volume of detectors is neglected. 
Ignoring this a priori information often leads to non-competitive reconstruction
methods. 

To avoid technicalities in this introduction, let us consider a discrete 
version where the exact data vector $g^{\dagger}$ belongs to $[0,\infty)^J$,
and $g_j^{\dagger}$ is the expected number of counts of the $j$th detector.
Then the observed count data are described by a vector 
$\gdelta\in \mathbb{N}_0^J$ of $J$ independent
Poisson distributed random variables with mean $g^{\dagger}$. 
A continuous version will be discussed in section \ref{sec:kl}.
Since 
$-\ln \prob[\gdelta\vert g]=
-\ln\left(\prod_j e^{-g_j} g_j^{\gdelta_j}(\gdelta_j!)^{-1}\right)
= \sum_j[g_j - \gdelta_j\ln g_j] + c$ with a constant $c$ independent of $g$
(except for the special cases specified in eq.~\eqref{eq:SforPoisson}),  
the negative log-likelihood data misfit functional  is given by 
\begin{equation}\label{eq:SforPoisson}
\S{g}{\gdelta} := \begin{cases}
\sum\limits_{j=1}^J \left[g_j- \gdelta_j\ln g_j\right],&
g\geq 0 \mbox{ and }\{j:\gdelta_j>0, g_j=0\}= \emptyset,\\
\infty, &\mbox{else,}
\end{cases}
\end{equation}
using the convention $0\ln 0:=0$. 
Setting $\gdelta = g^{\dagger}$ and subtracting the  minimal value   
$\sum_{j=1}^J \left[g^{\dagger}_j- g^{\dagger}_j\ln g^{\dagger}_j\right]$ attained at $g = g^{\dagger}$, we obtain a discrete version of the \emph{Kullback-Leibler divergence} 
\begin{equation}\label{eq:defi_KL}
\KL{g}{g^{\dagger}}:=
\begin{cases}
\sum\limits_{j=1}^J \left[g_j- g^{\dagger}_j -g^{\dagger}_j\ln\left(\frac{g_j}{g^\dagger_j}\right)\right]\,&
g\geq 0,\,\,\{j:g^{\dagger}_j>0, g_j=0\}= \emptyset,\\
\infty, &\mbox{else\,.}
\end{cases}
\end{equation}
Note that both $\calS$ and $\mathbb{KL}$ are convex in their second arguments.

\bigskip
A standard way to solve perturbed nonlinear operator equations 
\eqref{eq:opeq} is the Gau{\ss}-Newton method. If 
$F'$ denotes the Gateaux derivative of $F$,  it is given by given by 
$u_{n+1}:=\argmin_{u \in \mathfrak{B}}
\|F \left(u_n\right)+F'\left(u_n;u-u_n\right)-\gdelta\|^2$. 
As explained above, for data errors with a non-Gaussian distribution 
it is in general not appropriate to use a squared norm as data 
misfit functional. 
Therefore, we will consider general data misfit functionals
$\calS : \Zspace \times \Yspace \to \left(-\infty,\infty\right]$ 
where  $\Zspace$ is a space of (possibly discrete) observations $\gdelta$. 

Since inverse problems are typically ill-posed in the sense that 
$F$ and its derivatives $F'(u_n;\cdot)$ do not have continuous inverses,
regularization has to be used. Therefore, we add a 
proper convex penalty functional ${\mathcal R} : \Xspace \to \left(-\infty, \infty\right]$, which should be chosen to incorporate
\textit{a priori} knowledge about the unknown solution $u^\dagger$. 
This leads to the iteratively regularized Newton-type method 
\begin{subequations}\label{eqs:method}
\begin{equation}\label{eq:NMgen}
u_{n+1} := \argmin\limits_{u \in \mathfrak{B}} \left[\S{F \left(u_n\right)+F'\left(u_n;u-u_n\right)}{\gdelta} + \alpha_n \mathcal R \left(u\right)\right]
\end{equation}
which will be analyzed in this paper. The regularization parameters 
$\alpha_n$ are chosen such that 
\begin{equation}\label{eq:parac}
\alpha_0 \leq 1, \qquad \alpha_n \searrow 0, \qquad 1 \leq \frac{\alpha_n}{\alpha_{n+1}} \leq \Cdec \qquad \text{for all} \qquad n \in \mathbb N
\end{equation}
\end{subequations}
for some constant $\Cdec$, typically $\alpha_n=\alpha_0 \Cdec^{-n}$
with $\Cdec=3/2$. 

If $\Yspace = \mathbb{R}^J$, $F(u) = (F_j(u))_{j=1,...,d}$, 
and $\calS$ is given by \eqref{eq:SforPoisson}, 
we obtain the convex minimization problems
\begin{equation}\label{eq:NewtonKL}
\begin{aligned}
u_{n+1} := &\argmin\limits_{u \in \mathfrak{B}_n}
 \Big[
\sum\limits_{j=1}^J \big[F_j\left(u_n\right)+F_j'\left(u_n;u-u_n\right)- \\
&\qquad \qquad - \gdelta_j 
\ln (F_j\left(u_n\right)+F_j'\left(u_n;u-u_n\right))\big]+ \alpha_n \mathcal R \left(u\right)\Big]
\end{aligned}
\end{equation}
in each Newton step where $\mathfrak{B}_n:=\{u\in\mathfrak{B}~\big|~ \S{F(u)+F'(u_n;u-u_n)}{\gdelta}<\infty\}$. 
In principle, several methods for the solution of \eqref{eq:NewtonKL} are 
available. In particular we mention inverse scale space methods 
\cite{obgxy05,bsb11} for linear operator equations and 
total variation penalties $\mathcal R$. 
EM-type methods cannot readily be used for the solution of the
convex minimization problems \eqref{eq:NewtonKL}  
(or subproblems of the inverse scale space method as in \cite{bsb11}) 
if $F'(u_n;\cdot)$ is not positivity preserving as in our examples. 
A simple algorithm for the solution of subproblems of the type 
\eqref{eq:NewtonKL} is discussed in section \ref{sec:applications}. 
We consider the design of more efficient algorithms for minimizing 
the functionals \eqref{eq:NewtonKL} 
for large scale problems as an important problem for future research. 

\smallskip

The most common choice of the data misfit functional is $\S{g}{\hat g} = \left\Vert g-\hat g\right\Vert_{\Yspace}^2$ with a  
Hilbert space norm $\|\cdot\|_{\Yspace}$. This can be motivated by 
the case of (multi-variate) Gaussian errors. 
 If the penalty term is also given by a Hilbert space norm $\mathcal{R}\left(u\right)=\left\Vert u-u_0\right\Vert_{\Xspace}^2$,  
\eqref{eqs:method} becomes the iteratively regularized Gauss-Newton method (IRGNM)
which is one of the most popular methods for solving nonlinear ill-posed operator equations
\cite{b92,bns97,bk04,kns08}.  If the penalty term $\left\Vert u-u_0\right\Vert_{\Xspace}^2$ 
is replaced by $\left\Vert u - u_n\right\Vert_{\Xspace}^2$ one obtains the Levenberg-Marquardt method, 
which is well-known in optimization and has first been analyzed as regularization method
in \cite{hanke:97a}. 
Recently, a generalization of the IRGNM to Banach spaces has been proposed and analyzed 
by Kaltenbacher \& Hofmann \cite{kh10}.

\bigskip

As an alternative to \eqref{eqs:method} we mention Tikhonov-type or 
variational regularization methods of the form
\begin{equation}\label{eq:nltikh}
\widehat{u}_{\alpha}:=\argmin_{u\in \mathfrak{B}}\left[\S{F(u)}{\gdelta} + \alpha \mathcal{R}\left(u\right) \right]\,.
\end{equation}
Here $\alpha>0$ is a regularization parameter. For nonlinear operators this is in general a non-convex 
optimization problem even if $\S{\cdot}{\gdelta}$ and $\mathcal R$ are convex. Hence, \eqref{eq:nltikh} may have many 
local minima and it cannot be guaranteed that the global minimum can be found numerically.
Let us summarize some recent convergence results on this method: 
Bardsley \cite{b10} shows stability and convergence for linear operators and $\calS =\mathbb{KL}$. Benning \& Burger \cite{bb11} prove rates of convergence for linear operators under the special source condition $F^*\omega \in \partial \mathcal{R}(u^{\dagger})$. Generalizations to nonlinear operators
and general variational source conditions were published simultaneously by
Bot \& Hofmann \cite{bh10}, Flemming \cite{f10}, and Grasmair 
\cite{grasmair:10}.

Given some rule to choose the stopping index $n_*$ our main results
(Theorems \ref{thm:cr} and \ref{thm:cradd}) establish rates 
of convergence of the method 
\eqref{eqs:method}, i.e.\ uniform estimates
of the error of the final iterate in terms of some data noise level
$\err$
\begin{equation}\label{eq:final_conv_rate}
\left\Vert u_{n_*} - u^\dagger\right\Vert \leq C\varphi(\err)
\end{equation}
for some increasing, continuous function $\varphi:[0,\infty)\to[0,\infty)$
satisfying $\varphi(0)=0$. 
For the classical deterministic error model $\|\gdelta-g\|\leq \delta$ and $\S{g}{\gdelta}=\|g-\gdelta\|^r$ with some $r\geq 1$ we have $\err = \delta^r$. 
In this case we recover most of the known convergence results on the 
IRGNM for weak source conditions. 
Our main results imply error estimates for Poisson data provided 
a concentration inequality holds true. In this case $\err = \frac{1}{\sqrt{t}}$ where $t$ can be interpreted as an exposure time proportional to  the 
expected total number of photons, and an estimate of the form 
\eqref{eq:final_conv_rate} holds true with the right hand side replaced by an expected error. 

As opposed to a Hilbert or Banach space setting our data misfit 
functional $\calS$ does not necessarily fulfill a triangle inequality. Therefore, it is necessary to use more
general formulations of the noise level and the tangential cone condition, which controls the degree of
nonlinearity of the operator $F$. Both coincide with the usual assumptions if $\calS$ is given by a norm.  
Our analysis uses variational methods rather than methods based on spectral theory, which have recently been studied in the context
of inverse problems by a number of authors (see, e.g., \cite{bo04,rs06,s08,hkps07,kh10}).

The plan of this paper is as follows: In the following section we formulate our first main convergence theorem
(Theorem \ref{thm:cr}) and discuss its assumptions. The proof will be given 
in section \ref{sec:convergence}. In the following section \ref{sec:lepskii} we discuss the case of additive variational inequalities and state a convergence rates result for a Lepski{\u\i}-type stopping rule (Theorem \ref{thm:cradd}). 
In section \ref{sec:spec} we compare our result to previous results on the iteratively regularized Gauss-Newton method. Section \ref{sec:kl} is devoted to the special case of Poisson data, which has been our main motivation.
We conclude our paper with numerical results for an inverse obstacle scattering problem
and a phase retrieval problem in optics in section \ref{sec:applications}.

\section{Assumptions and convergence theorem with a priori stopping rule}\label{sec:assumptions}

Throughout the paper we assume the following mapping and differentiability properties of the forward operator $F$:
\begin{ass}[Assumptions on $F$ and  $\mathcal R$]\label{ass:F}
Let $\Xspace$ and $\Yspace$ be Banach spaces and let $\mathfrak{B}\subset \Xspace$ a convex subset. \\[0.1cm]
Assume that the forward operator $F:\mathfrak{B}\to \Yspace$ and the penalty functional 
$\mathcal R: \Xspace \to \left(-\infty, \infty\right]$ have the following properties:
\begin{enumerate}
\item $F$ is injective.
\item $F:\mathfrak{B}\to \Yspace$ is continuous, the first variations
$F'(u;v-u):=\lim_{t\searrow 0} \frac{1}{t}(F(u+t(v-u))-F(u))$ exist for
all $u,v\in\mathfrak{B}$, and $h\mapsto F'(u;h)$ can be extended to a 
bounded linear operator $F'[u]\in  L(\Xspace,\Yspace)$ 
for all $u\in\mathfrak{B}$.
\item  $\mathcal{R}$ is proper and convex.
\end{enumerate}
\end{ass}
At interior points  $u \in \mathfrak B$ the second assumption amounts to 
Gateaux differentiability of $F$. 

\bigskip

To motivate our assumptions on the data misfit functional, let us consider the case that 
$\gdelta = F(u^{\dagger})+\xi$, and $\xi$ is Gaussian white noise on the
Hilbert space $\mathcal{Y}$, i.e.\ $\langle\xi,g\rangle\sim N(0,\|g\|^2)$ and
$\EW \langle \xi,g\rangle \,\langle\xi,\tilde{g}\rangle = 
\langle g, \tilde{g}\rangle$
for all $g,\tilde{g}\in\mathcal{Y}$. If $\Yspace=\mathbb{R}^J$, 
then the negative log-likelihood functional is given by 
$\S{g}{\gdelta} = \|g-\gdelta\|_{2}^2$. However, 
in an infinite dimensional Hilbert space $\Yspace$ we have $\|\gdelta\|_{\Yspace}=\infty$ almost surely, 
and $\S{\cdot}{\gdelta}\equiv \infty$ is obviously not a useful data misfit term. 
Therefore, one formally subtracts $\|\gdelta\|_{\Yspace}^2$ (which is independent of $g$) to obtain 
$\S{g}{\gdelta} := \left\Vert g\right\Vert_{\Yspace}^2 - 2 \left<\gdelta,g\right>_{\Yspace}$.
For exact data $g^{\dagger}$ we can of course use the data misfit functional $\T{g}{g^\dagger} = \left\Vert g-g^\dagger\right\Vert_{\Yspace}^2$. As opposed to $\calS$,  the functional $\calT$ is nonnegative and does indeed describe 
the size of the error in the data space $\Yspace$. 
It will play an important role in our analysis. 

It may seem cumbersome to work with two different types data misfit functionals $\calS$ and $\calT$, and a straightforward idea to fix the free additive constant in $\calS$ is to introduce $\Stil{g}{\gdelta}:= \S{g}{\gdelta}-\tilde{\mathfrak{s}}$ with 
$\tilde{\mathfrak{s}}:=\inf_g \S{g}{\gdelta}$. 
Then we obtain indeed that $\Stil{g}{g^{\dagger}}=\T{g}{g^{\dagger}}$. 
However, the expected error $\EW \big|\S{g}{\gdelta}-\mathfrak{s}-\T{g}{g^{\dagger}}\big|^2$ 
is not minimized for $\mathfrak{s} = \tilde{\mathfrak{s}}$,
but for 
$\mathfrak{s}= \EW \S{g}{\gdelta}-\T{g}{g^{\dagger}} = -\|g^{\dagger}\|^2$. 
Note that $\mathfrak{s}$ depends on the unknown $g^{\dagger}$, but this
does not matter since the value of $\mathfrak{s}$ does not affect the
numerical algorithms.  
For this choice of $\mathfrak{s}$ the error has the
convenient representation  
$\S{g}{\gdelta}+ \|g^{\dagger}\|^2-\T{g}{g^{\dagger}} = 
-2\langle \xi,g\rangle_{\Yspace}$.
Bounds on $\sup_{g\in\tilde{\mathcal{Y}}}\left|\langle \xi,g\rangle_{\Yspace}\right|$ with high probabilities 
for certain subsets $\tilde{\mathcal{Y}}\subset \mathcal{Y}$ (concentration 
inequalities) have been studied intensively in probability theory 
 (see e.g.~\cite{massart:07}). Such results can be used in case of Gaussian errors to show that
the following deterministic error assumption holds true with high probability
and uniform bounds on $\err(g)$ for $g\in\tilde{\mathcal{Y}}$. 

\begin{ass}[data errors, properties of $\calS$ and $\calT$]\label{ass:SR}
Let $u^\dagger \in \mathfrak B \subset \Xspace$ be the exact solution and denote by 
$g^\dagger := F\left(u^\dagger\right) \in \Yspace$ the exact data. Let $\Zspace$ be a set containing all possible observations and $\gdelta\in \Zspace$ 
the observed data. Assume that:
\begin{enumerate}
\item The fidelity term $\calT: F\left(\mathfrak{B}\right) \times \Yspace \to [0,\infty]$ with respect to exact data fulfills 
$\T{g^\dagger}{g^\dagger}=0$. 
\item $\calT$ and the fidelity term $\calS: \Zspace \times \Yspace \to (-\infty,\infty]$ with respect to noisy data are connected as follows: There exists a constant $\Cerr\geq 1$ and functionals $\err : \Yspace \to \left[0, \infty\right]$ and $\mathfrak{s}:F \left(\mathfrak B\right) \to (-\infty,\infty)$ such that 
\begin{subequations}\label{eq:err}
\begin{eqnarray}
\S{g}{\gdelta} - \mathfrak{s}(g^{\dagger})&\leq& \Cerr \T{g}{g^\dagger} + \Cerr\err \left(g\right)\\
\T{g}{g^\dagger} &\leq& 
\Cerr \paren{\S{g}{\gdelta} - \mathfrak{s}(g^{\dagger})}
+ \Cerr \err \left(g\right)
\end{eqnarray}
\end{subequations}
for all $g \in \Yspace$.
\end{enumerate}
\end{ass}
\begin{example}\label{ex:errors}
\begin{enumerate}
	\item \emph{Additive deterministic errors in Banach spaces.}
Assume that $\Zspace = \Yspace$, 
\[
\|\gdelta-g^{\dagger}\|\leq \delta,\qquad\mbox{and}\qquad 
\S{g_1}{g_2} = \T{g_1}{g_2}= \left\Vert g_1-g_2\right\Vert_{\Yspace}^r
\]
with $r\in\left[1,\infty\right)$. 
Then it follows from the simple inequalities $\left(a+b\right)^r \leq 2^{r-1}\left(a^r+b^r\right)$ and $\left|a-b\right|^r+b^r\geq 2^{1-r}a^r$
that \eqref{eq:err} holds true with $\err\equiv \left\Vert\gdelta-g^{\dagger}\right\Vert_{\Yspace}^r$, $\mathfrak{s} \equiv 0$ and $\Cerr = 2^{r-1}$.  
\item For randomly perturbed data 
a general recipe for the choice of $\mathcal{S}, \mathcal{T}$ and $\mathfrak{s}$
is to define $\mathcal{S}$ as the log-likelihood functional, 
$\mathfrak{s}(g^{\dagger}):= \EW_{g^{\dagger}} \S{g^{\dagger}}{\gdelta}$ and
$\T{g}{g^{\dagger}}:=\EW_{g^{\dagger}}\S{g}{\gdelta}-\mathfrak{s}(g^{\dagger})$. 
Then we always have $\T{g^{\dagger}}{g^{\dagger}}=0$, but part 2.
of Assumption \ref{ass:SR} has to be verified case by case. 
\item \emph{Poisson data.} For discrete Poisson data we have already seen
in the introduction that the general recipe of the previous point yields
$\calS$ given by \eqref{eq:SforPoisson}, $\calT = \mathbb{KL}$
and $\mathfrak{s}(g^{\dagger})=\sum_{j=1}^J \left[g^{\dagger}_j- g^{\dagger}_j\ln\left(g^{\dagger}_j\right)\right]$. It is 
easy to see that $\KL{g}{g^{\dagger}}\geq 0$ for all $g^{\dagger}$ and $g$.
Then \eqref{eq:err} holds true with $\Cerr = 1$ and
\[
\err(g) = \begin{cases}
\Big|\sum\limits_{j=1}^J \ln\left(g_j\right) \left(\gdelta_j-g^{\dagger}_j\right)\Big|,& g\geq 0, \{j:g_j=0,
g^{\dagger}_j+\gdelta_j>0\} = \emptyset\\
\infty,&\mbox{else\,.}
\end{cases}
\]
Obviously, it will be necessary to show that 
$\err \left(g\right)$ is finite and even small in some sense
for all $g$ for which the inequalities \eqref{eq:err} 
are applied (see section \ref{sec:kl}). 
\end{enumerate}
\end{example}

To simplify our notation we will assume in the following analysis 
that $\mathfrak{s}\equiv 0$ or equivalently replace $\S{g}{\gdelta}$
by $\S{g}{\gdelta}- \mathfrak{s}(g^{\dagger})$. 
As already mentioned in the motivation of Assumption \ref{ass:SR}, it is not relevant that $\mathfrak{s}(g^{\dagger})$ is unknown
since the value of this additive constant does not influence the iterates
$u_n$ in \eqref{eq:NMgen}.

Typically $\mathcal{S}$ and $\mathcal{T}$ will be convex in their second
arguments, but we do not need this property in our analysis. However,
without convexity it is not clear if the numerical solution of \eqref{eq:NMgen}
is easier than the numerical solution of \eqref{eq:nltikh}. 

\bigskip

\begin{ass}[Existence]\label{ass:ex}
For any $n \in \mathbb N$ the problem \eqref{eq:NMgen} has a solution.
\end{ass}
\begin{rem}
By standard arguments the following properties are sufficient to ensure existence of a solution to \eqref{eq:NMgen} for convex $\S{\cdot}{\gdelta}$ (see \cite{hkps07,p08,f10}):\\[0.1cm]
There are possibly weaker topologies $\tau_{\Xspace}$, $\tau_{\Yspace}$ on $\Xspace, \Yspace$ respectively such that
\begin{enumerate}
\item $\mathfrak B$ is sequentially closed w.r.t. $\tau_{\Xspace}$,
\item $F'\left(u;\cdot\right)$ is sequentially continuous w.r.t. $\tau_{\Xspace}$ and $\tau_{\Yspace}$ for all $u\in \mathfrak B$, 
\item the penalty functional $\mathcal R: \Xspace \to \left(-\infty, \infty\right]$ is sequentially lower semi-continuous with respect to $\tau_{\Xspace}$,
\item the sets $M_{\mathcal R}\left(c\right) := \left\{u \in \Xspace ~\big|~\mathcal R \left(u\right) \leq c\right\}$ are sequentially pre-compact with respect to $\tau_{\Xspace}$ for all $c \in \mathbb R$ and
\item for each $\gdelta$ the data misfit term $\S{\cdot}{\gdelta} : \Yspace \to \left(-\infty,\infty\right]$ is sequentially lower semi-continuous w.r.t. $\tau_{\Yspace}$.
\end{enumerate}
Note that for our analysis we do not require that the solution to \eqref{eq:NMgen} is unique or 
depends continuously on the data $\gdelta$ even though these properties are desirable
for other reasons. 
Obviously, uniqueness is given if $\calS$ is convex and $\mathcal R$ 
is strictly convex, and 
there are reasonable assumptions on $\calS$ which guarantee continuous dependence, cf.\ \cite{p08}.
\end{rem}

\bigskip

All known convergence rate results for nonlinear ill-posed problems under weak source conditions
assume some condition restricting the degree of nonlinearity of the operator $F$. 
Here we use a generalization of the tangential cone condition which was
introduced in \cite{HNS:95} and is frequently used for the analysis of regularization
methods for nonlinear inverse problems. It must be said, however, that for many
problems it is very difficult to show that this condition is satisfied
(or not satisfied). Since $\calS$ does not necessarily 
fulfill a triangle inequality we have to use a generalized formulation
of the tangential cone condition, which follows from the standard formulation
if $\calS$ is given by the power of a norm (cf.\ Lemma \ref{lem:tan_cone}).
\begin{ass}[Generalized tangential cone condition]\label{ass:nl}$\;$
\begin{subequations}\label{eqs:tc}
\begin{enumerate}
\item[(A)]
There exist constants $\eta$ (later assumed to be sufficiently small) and $\Ctc \geq 1$ such that
for all $\gdelta \in \Zspace$
\begin{align}
&\frac{1}{\Ctc} \S{F\left(v\right)}{\gdelta} - \eta \S{F\left(u\right)}{\gdelta} \nonumber\\[0.1cm]
\label{eq:tca}
\leq& \S{F \left(u\right)  + F'\left(u;v-u\right)}{\gdelta}\\[0.1cm]
\leq &\Ctc \S{F\left(v\right)}{\gdelta} + \eta \S{F\left(u\right)}{\gdelta}
\qquad \mbox{for all }u,v\in\mathfrak{B}.\nonumber
\end{align}
\item[(B)] There exist constants $\eta$ (later assumed to be sufficiently small) and $\Ctc \geq 1$ such that
\begin{align}
&\frac{1}{\Ctc} \T{F\left(v\right)}{g^{\dagger}} - \eta \T{F\left(u\right)}{g^{\dagger}} \nonumber\\[0.1cm]
\label{eq:tcb}
\leq& \T{F \left(u\right)  + F'\left(u;v-u\right)}{g^{\dagger}}\\[0.1cm]
\leq &\Ctc \T{F\left(v\right)}{g^{\dagger}}+ \eta \T{F\left(u\right)}{g^{\dagger}}
\qquad \mbox{for all }u,v\in\mathfrak{B}.\nonumber
\end{align}
\end{enumerate}
\end{subequations}
\end{ass}
This condition ensures that the nonlinearity of $F$ fits together with the data misfit functionals $\calS$ or $\calT$.
Obviously, it is fulfilled with $\eta = 0$ and $\Ctc = 1$ if $F$ is linear.

\bigskip
It is well-known that for ill-posed problems rates of convergence can only be obtained under an
additional ''smoothness condition'' on the solution 
(see \cite[Prop. 3.11]{ehn96}). 
In  a Hilbert space setting 
such conditions are usually formulated as source conditions in the form 
\begin{equation}\label{eq:sc}
u^{\dagger}-u_0 = \varphi\left(F'\left[u^\dagger\right]^*F'\left[u^\dagger\right]\right) \omega
\end{equation}
for some $\omega\in \mathcal{X}$ where $\varphi:[0,\infty)\to [0,\infty)$ is a so-called \emph{index function},
i.e.\ $\varphi$ is continuous and monotonically increasing with $\varphi(0)=0$. Such general source
conditions were systematically studied in \cite{hegland:95, MP:03}. 
The most common choices of $\varphi$
are discussed in section \ref{sec:spec}. 

To formulate similar source conditions in Banach spaces, we first have to introduce 
Bregman distances, which will also be used to measure the error of our approximate solutions
(see \cite{bo04}):
Let $u^* \in \partial \mathcal  R \left(u^\dagger\right)$ be a subgradient (e.g.\ $u^*=u^{\dagger}-u_0$
if $\mathcal{R}(u) = \frac{1}{2}\|u-u_0\|^2$ with a Hilbert norm $\|\cdot\|$). 
Then the Bregman distance of $\mathcal R$ between $u$ and $u^\dagger$ is given by 
\[
\mathcal D^{u^*}_{\mathcal R} \left(u, u^\dagger\right) := \mathcal R \left(u\right) - \mathcal R \left(u^\dagger\right) - \left<u^*, u-u^\dagger\right>.
\]
If $\Xspace$ is a Hilbert space and $\mathcal{R}(u) = \frac{1}{2}\|u-u_0\|^2$, 
we have $\mathcal D^{u^*}_{\mathcal R} \left(u, u^\dagger\right) 
= \frac{1}{2}\|u-u^{\dagger}\|^2$. Moreover, if $\Xspace$ is a $q$-convex Banach space ($1 < q \leq 2$) and $\mathcal R \left(u\right) = \left\Vert u \right\Vert^q$, then there exists a constant $\Cbreg >0$ such that
\begin{equation}\label{eq:lower_bound_breg}
\left\Vert u-u^\dagger\right\Vert^q \leq \Cbreg \mathcal D^{u^*}_{\mathcal R} \left(u, u^\dagger\right)
\end{equation}
for all $u \in \Xspace$ (see e.g. \cite{bkmss08}). In those cases, convergence rates w.r.t. the Bregman distance also imply rates w.r.t. the Banach space norm.

Now we can formulate the following variational formulation of the source condition \eqref{eq:sc},
which is a slight variation of the one proposed in \cite{kh10}:
\begin{assa}[Multiplicative variational source condition]\label{ass:sc}
There exists $u^*\in\partial \mathcal R \left(u^\dagger\right) \subset \Xspace'$, $\beta \geq 0$ and a concave index function $\varphi : \left(0,\infty\right) \to \left(0,\infty\right)$ such that
\begin{equation}\label{eq:scgen}
\left<u^*, u^\dagger-u\right> \leq \beta\mathcal D_{\mathcal R}^{u^*} \left(u, u^\dagger\right)^{\frac12} \varphi\left(\frac{\T{F\left(u\right)}{g^\dagger}}{\mathcal D_{\mathcal R}^{u^*} \left(u, u^\dagger\right)}\right)\qquad 
\mbox{for all }u \in \mathfrak B.
\end{equation}
Moreover, we assume that
\begin{equation}\label{eq:condvarphi}
t\mapsto \frac{\varphi\left(t\right)}{\sqrt{t}} \qquad\text{is monotonically decreasing}.
\end{equation}
\end{assa}
As noted in \cite{kh10} using Jensen's inequality, 
a Hilbert space source condition \eqref{eq:sc} for which
$\left(\varphi^2\right)^{-1}$ is convex
implies the variational inequality 
\begin{equation}\label{eq:schspace}
\left|\left<u^*, u-u^\dagger\right>\right| \leq \left\Vert \omega\right\Vert \left\Vert u-u^\dagger\right\Vert \varphi\left(\frac{\left\Vert F'\left[u^\dagger\right]\left(u-u^\dagger\right)\right\Vert^2}{\left\Vert u-u^\dagger\right\Vert^2}\right).
\end{equation}
The tangential cone condition now shows that an inequality of 
type \eqref{eq:scgen} is valid and hence, in a Hilbert space setup Assumption~\ref{ass:sc} is weaker than \eqref{eq:sc} at least for linear operators. As opposed to \cite{kh10} we have omitted
absolute values on the left hand side of \eqref{eq:scgen} since they are not
needed in the proofs, and this form may allow for better index functions $\varphi$
if $u^{\dagger}$ is on the boundary of $\mathfrak{B}$.

In many recent publications \cite{s08,bh10,hy10,f10} variational source conditions 
in additive rather than multiplicative form have been used. Such conditions will 
be discussed in section \ref{sec:lepskii}.

Since we use a source condition with a general index function $\varphi$, we need to restrict the nonlinearity of $F$ with the help of a tangential cone condition. Nevertheless, we want to mention that for $\varphi \left(t\right) = t^{1/2}$ in \eqref{eq:scgen} our convergence analysis also works under a generalized Lipschitz assumption, but this lies beyond the aims of this paper. The cases $\varphi \left(t\right) = t^\nu$ with $\nu > \frac12$ where similar results are expected are not covered by Assumption~\ref{ass:sc}, since for the motivation in the Hilbert space setup we needed to assume that $\left(\varphi^2\right)^{-1}$ is convex, which is not the case for $\nu > \frac12$. 

\smallskip

In our convergence analysis we will use the following two functions,
which are both index functions as well as their inverses:
\begin{align}\label{eqs:defiThetaf}
\begin{aligned}
\Theta \left(t\right) &:= t\varphi^2\left(t\right), \\
\vartheta\left(t\right) &:= \sqrt{\Theta\left(t\right)} = \sqrt{t}\varphi \left(t\right)
\end{aligned}
\end{align}
\medskip

We are now in a position to formulate our convergence result with a priori stopping rule:

\begin{thm}\label{thm:cr}
Let Assumption~\ref{ass:F}, \ref{ass:SR}, \ref{ass:ex}, \ref{ass:nl}A or \ref{ass:nl}B and \ref{ass:sc}A hold true, and  
suppose that 
$\eta$, $\mathcal D_{\mathcal R}^{u^*} \left(u_0, u^\dagger\right)$ and $\T{F\left(u_0\right)}{g^\dagger}$ are sufficiently small. 
Then the iterates $u_n$ defined by \eqref{eqs:method} with exact data $\gdelta = g^{\dagger}$ fulfill
\begin{subequations}\label{eq:cr_ex_data}
\begin{align}
\mathcal D^{u^*}_{\mathcal R} \left(u_n, u^\dagger\right) &= \mathcal O \left(\varphi^2 \left(\alpha_n\right)\right),\\
\T{F \left(u_n\right)}{g^\dagger} &= \mathcal O \left(\Theta \left(\alpha_n\right)\right)
\end{align}
\end{subequations}
as $n \to \infty$. For noisy data define 
\begin{subequations}\label{eq:defi_err}
\begin{align}\label{eq:defi_errnA}
\err_{n} := \frac{1}{\Cerr}\err\left(F\left(u_{n+1}\right)\right) + 2\eta \Ctc \err \left(F\left(u_n\right)\right)+\Ctc\Cerr\err\left(g^\dagger\right)
\end{align}
in case of Assumption~\ref{ass:nl}A or
\begin{equation}\label{eq:defi_errnB}
\begin{array}{rcl} \err_{n} &:= &\err \left(F\left(u_n\right)+ F'\left(u_n;u_{n+1}-u_n\right)\right) \\[0.1cm]
&&+ \Cerr\err \left(F\left(u_n\right)+ F'\left(u_n;u^\dagger-u_n\right)\right)\end{array}
\end{equation}
\end{subequations}
under Assumption~\ref{ass:nl}B, and choose the stopping index $n_*$ by
\begin{equation}\label{eq:stopr}
n_*   := \min\left\{n \in \mathbb N ~\big|~ \Theta\left(\alpha_n\right) \leq \tau \err_n\right\}
\end{equation}
with a sufficiently large parameter $\tau\geq 1$. Then \eqref{eq:cr_ex_data} holds for $n \leq n_*$ and the following convergence rates are valid:
\begin{subequations}\label{eqs:conv_in_XY}
\begin{align}
\label{eq:conv_in_X}
\mathcal D^{u^*}_{\mathcal R} \left(u_{n_*}, u^\dagger\right) 
&= \mathcal O \left(\varphi^2 \left(\Theta^{-1} \left(\err_{n_*}\right)\right)\right),\\[0.1cm]
\label{eq:conv_in_Y}
\T{F \left(u_{n_*}\right)}{g^\dagger} &= \mathcal O \left(\err_{n_*}\right).
\end{align}
\end{subequations}
\end{thm}

\section{Proof of Theorem \ref{thm:cr}}\label{sec:convergence}
We will split the proof into to two main parts. For brevity we will denote
\begin{align}\label{eq:abbrev}
d_n &:=\mathcal D^{u^*}_{\mathcal R} \left(u_n, u^\dagger\right)^{\frac12}, \\[0.1cm]
s_n&:= \T{F\left(u_n\right)}{g^\dagger}.
\end{align}
Let us now start with the following
\begin{lem}\label{lem:2}
Let the assumptions of Theorem \ref{thm:cr} hold true. Then we have a \textbf{recursive error estimate} of the form
\begin{subequations}\label{eq:ree}
\begin{equation}\label{eq:p6}
\alpha_n d_{n+1}^2 + \frac{1}{\Ctc\Cerr} s_{n+1} \leq \eta \left(\Cerr + \frac{1}{\Cerr}\right) s_n + \alpha_n \beta d_{n+1} \varphi \left(\frac{s_{n+1}}{d_{n+1}^2}\right) +\err_n
\end{equation}
in the case of \ref{ass:nl}B and
\begin{equation}\label{eq:p12}
\alpha_n d_{n+1}^2 + \frac{1}{\Ctc\Cerr}s_{n+1} \leq 2\eta \Cerr s_n + \alpha_n \beta d_{n+1} \varphi \left(\frac{s_{n+1}}{d_{n+1}^2}\right) + \err_n
\end{equation}
\end{subequations}
in the case of \ref{ass:nl}A for all $n \in \mathbb N$.
\end{lem}
\begin{proof}
Due to \eqref{eq:scgen} we have
\begin{align}
\mathcal R \left(u_{n+1}\right) - \mathcal R \left(u^\dagger\right) &= \mathcal D^{u^*}_{\mathcal R} \left(u_{n+1},u^\dagger\right) - \left< u^*, u^\dagger - u_{n+1}\right> \nonumber\\[0.1cm]
& \geq d_{n+1}^2  -\beta d_{n+1}\varphi\left(\frac{s_{n+1}}{d_{n+1}^2}\right).\label{eq:p1}
\end{align}
From the minimality condition \eqref{eq:NMgen} with $u=u^{\dagger}$
we obtain 
\begin{align}
&\alpha_n \left(\mathcal R\left(u_{n+1}\right) - \mathcal R \left(u^\dagger\right) \right) + \S{F\left(u_n\right) + F'\left(u_n;u_{n+1}- u_n\right)}{\gdelta}\nonumber\\[0.1cm]
\leq &\S{F\left(u_n\right) +F'\left(u_n;u^\dagger-u_n\right)}{\gdelta}\, ,\label{eq:p2}
\end{align}
and putting \eqref{eq:p1} and \eqref{eq:p2} together we find that
\begin{align}
&\alpha_n d_{n+1}^2 + \S{F\left(u_n\right) + F'\left(u_n;u_{n+1}- u_n\right)}{\gdelta} \nonumber\\[0.1cm]
\leq& \S{F\left(u_n\right)+ F'\left(u_n;u^\dagger-u_n\right)}{\gdelta}+  \alpha_n \beta d_{n+1}\varphi\left(\frac{s_{n+1}}{d_{n+1}^2}\right). \label{eq:p11}
\end{align}
\begin{itemize}
\item In the case of \ref{ass:nl}B we use \eqref{eq:err}, which yields
\begin{align*}
&\alpha_n d_{n+1}^2 + \frac{1}{\Cerr}\T{F\left(u_n\right) + F'\left(u_n;u_{n+1}- u_n\right)}{g^\dagger} \\[0.1cm]
\leq& \Cerr\T{F\left(u_n\right)+ F'\left(u_n;u^\dagger-u_n\right)}{g^\dagger}+ \alpha_n\beta d_{n+1}\varphi\left(\frac{s_{n+1}}{d_{n+1}^2}\right)+\err_{n}
\end{align*}
and \eqref{eq:tcb} with $v = u^\dagger$, $u = u_n$ leads to
\begin{align*}
&\alpha_n d_{n+1}^2 + \frac{1}{\Cerr}\T{F\left(u_n\right) + F'\left(u_n;u_{n+1}- u_n\right)}{g^\dagger}\\[0.1cm]
\leq& \eta \Cerr s_n + \alpha_n\beta d_{n+1}\varphi\left(\frac{s_{n+1}}{d_{n+1}^2}\right)+\err_{n}.
\end{align*}
By \eqref{eq:tcb} with $v = u_{n+1}$, $u = u_n$ we obtain \eqref{eq:p6}. 
\item In the case of \ref{ass:nl}A we are able to apply \eqref{eq:tca} with $v = u^\dagger$, $u = u_n$ and \eqref{eq:tca} with $v = u_{n+1}$ and $u = u_n$ to \eqref{eq:p11} to conclude
\begin{align*}
&\alpha_n d_{n+1}^2 + \frac{1}{\Ctc}\S{F\left(u_{n+1}\right)}{\gdelta} \\[0.1cm]
\leq& 2 \eta \S{F\left(u_n\right)}{\gdelta}+ \Ctc \S{F\left(u^\dagger\right)}{\gdelta} +  \alpha_n \beta d_{n+1}\varphi\left(\frac{s_{n+1}}{d_{n+1}^2}\right).
\end{align*}
Due to \eqref{eq:err} and Assumption~\ref{ass:SR}.2 this yields \eqref{eq:p12}.
\end{itemize}
\end{proof}

Before we deduce the convergence rates from the recursive error estimates \eqref{eq:ree} respectively, 
we note some inequalities for the index functions defined in \eqref{eqs:defiThetaf} and their inverses:

\begin{rem}
\begin{enumerate}
\item
We have
\begin{align}
\label{eq:varphif}
\varphi \left(\vartheta^{-1}\left(Ct\right)\right) &\leq \max\left\{\sqrt{C},1\right\} \varphi \left(\vartheta^{-1}\left(t\right)\right)\\
\label{eq:varphi2Theta}
\varphi^2 \left(\Theta^{-1}\left(Ct\right)\right) &\leq \max\left\{\sqrt{C},1\right\} \varphi ^2\left(\Theta^{-1}\left(t\right)\right)
\end{align}
for all $t \geq 0$ and $C>0$ if defined, where each inequality follows from
two applications of the monotonicity assumption \eqref{eq:condvarphi}
(see \cite[Remark 2]{kh10}). 
\item
Since $\varphi$ is concave, we have
\begin{equation}\label{eq:condvarphi2}
\varphi \left(\lambda t\right)  \leq \lambda \varphi \left(t\right) \qquad \text{ for all }t\text{ sufficiently small and }\lambda\geq 1
\end{equation}
\item
\eqref{eq:condvarphi2} implies the following inequality 
for all $t$ sufficiently small and $\lambda \geq 1$:
\begin{equation}\label{eq:Theta}
\Theta \left(\lambda t\right) \leq \lambda^3 \Theta \left(t\right)
\end{equation}
\end{enumerate}
\end{rem}

The following induction proof follows along the lines of 
a similar argument in the proof of \cite[Theorem 1]{kh10}:
\begin{lem}\label{lem:1}
Let the assumptions of Theorem \ref{thm:cr} hold. Then an estimate of the kind \eqref{eq:p6} implies 
\begin{align}
d_n &\leq C_1 \varphi \left(\alpha_n\right), \label{eq:ind1} \\[0.1cm]
s_n & \leq C_2 \Theta \left(\alpha_n\right) \label{eq:ind2}
\end{align}
for all $n \leq n_*$ in case of noisy data and for all $n \in \mathbb N$ in case of exact data where (due to $\eta$ sufficiently small)
\begin{align*}
C_2 &= \max\left\{4 \beta^2 \left(\Ctc\Cerr\Cdec\right)^3, \frac{2\Ctc\Cerr\Cdec^3}{\tau \left(1-2\Cdec^3 \Ctc\Cerr\eta \left(\Cerr + \frac{1}{\Cerr}\right)\right)}\right\}, \\[0.1cm]
C_1 &= \max\left\{\sqrt{2\beta}\sqrt[4]{C_2}, \sqrt{2\left(\eta C_2 \left(\Cerr + 1/\Cerr\right) + 1/\tau\right)}\Cdec\right\}.
\end{align*}
Since \eqref{eq:p12} is of the same form as \eqref{eq:p6} (only the constants differ), \eqref{eq:ind1} and \eqref{eq:ind2} are (with slightly changed constants) also valid under \eqref{eq:p12}.
\end{lem}
\begin{proof}
For $n = 0$ \eqref{eq:ind1} and \eqref{eq:ind2} are guaranteed by the assumption that $d_0$ and $s_0$ are small enough. For the induction step we observe that \eqref{eq:p6} together with \eqref{eq:stopr} and the induction hypothesis for $n \leq n_* -1$ implies
\[
\alpha_n d_{n+1}^2 + \frac{1}{\Ctc\Cerr} s_{n+1} \leq C_{\eta, \tau} \Theta \left(\alpha_n\right) + \alpha_n \beta d_{n+1} \varphi \left(\frac{s_{n+1}}{d_{n+1}^2}\right)
\]
where $C_{\eta, \tau} = \eta C_2 \left(\Cerr + 1/\Cerr\right) + 1/\tau$. Now we distinguish between two cases:\\[0.1cm]
\textbf{Case 1:} $\alpha_n \beta d_{n+1} \varphi \left(\frac{s_{n+1}}{d_{n+1}^2}\right) \leq C_{\eta, \tau} \Theta \left(\alpha_n\right)$.\\
In that case we find
\[
\alpha_n d_{n+1}^2 + \frac{1}{\Ctc\Cerr} s_{n+1} \leq 2 C_{\eta, \tau} \Theta \left(\alpha_n\right)
\]
which by $\Theta\left(t\right) / t = \varphi^2\left(t\right)$, \eqref{eq:condvarphi2} and \eqref{eq:Theta} implies
\begin{align*}
d_{n+1} &\leq \sqrt{2C_{\eta, \tau}}\varphi \left(\alpha_n\right) = \sqrt{2C_{\eta, \tau}}\varphi \left(\frac{\alpha_n}{\alpha_{n+1}} \alpha_{n+1}\right) \leq \sqrt{2C_{\eta, \tau}}\Cdec\varphi\left(\alpha_{n+1}\right),\\[0.1cm]
s_{n+1} &\leq 2\Ctc\Cerr C_{\eta, \tau}\Theta \left(\alpha_n\right)\leq 2\Ctc\Cerr C_{\eta, \tau}\Cdec^3\Theta \left(\alpha_{n+1}\right).
\end{align*}
The assertions now follow by $\sqrt{2C_{\eta, \tau}}\Cdec \leq C_1$ and $2\Ctc\Cerr C_{\eta, \tau}\Cdec^3 \leq C_2$ which is ensured by the definition of $C_2$.

\textbf{Case 2:} $\alpha_n \beta d_{n+1} \varphi \left(\frac{s_{n+1}}{d_{n+1}^2}\right) > C_{\eta, \tau} \Theta \left(\alpha_n\right)$.\\
In that case we find
\[
\alpha_n d_{n+1}^2 + \frac{1}{\Ctc\Cerr} s_{n+1} \leq 2 \alpha_n \beta d_{n+1} \varphi \left(\frac{s_{n+1}}{d_{n+1}^2}\right).
\]
If $d_{n+1} = 0$, then this implies $s_{n+1} = 0$ and hence the assertion is trivial. By multiplying with $\sqrt{s_{n+1}}$ and dividing by $d_{n+1}^2$ we have
\begin{equation}\label{eq:p8}
\alpha_n \sqrt{s_{n+1}} + \frac{1}{\Ctc\Cerr} \frac{s_{n+1}}{d_{n+1}^2} \sqrt{s_{n+1}} \leq 2 \beta \alpha_n \vartheta\left(\frac{s_{n+1}}{d_{n+1}^2}\right).
\end{equation}
Considering only the first term on the left hand side of \eqref{eq:p8} this is
\begin{equation}\label{eq:p9}
\vartheta^{-1}\left(\frac{\sqrt{s_{n+1}}}{2\beta}\right) \leq \frac{s_{n+1}}{d_{n+1}^2}
\end{equation}
and by considering only the second term on the left hand side of \eqref{eq:p8}
\begin{equation}\label{eq:p10}
\Phi \left(\frac{s_{n+1}}{d_{n+1}^2}\right) \sqrt{s_{n+1}} \leq 2 \beta \Ctc\Cerr\alpha_n
\end{equation}
where $\Phi \left(t\right) = \sqrt{t} / \varphi \left(t\right) = t / \vartheta\left(t\right)$. Plugging \eqref{eq:p9} into \eqref{eq:p10} using the monotonicity of $\Phi$ by \eqref{eq:condvarphi} we find
\[
\Phi \left(\vartheta^{-1}\left(\frac{\sqrt{s_{n+1}}}{2 \beta}\right)\right) \sqrt{s_{n+1}}\leq 2\beta \Ctc\Cerr \alpha_n.
\]
Since $\Phi \left(\vartheta^{-1}\left(t\right)\right) = \vartheta^{-1}\left(t\right) / t$ this shows
\[
\vartheta^{-1}\left(\frac{\sqrt{s_{n+1}}}{2\beta}\right)\leq \Ctc\Cerr \alpha_n.
\]
Hence,
\[
s_{n+1} \leq 4\beta^2 \Theta \left(\Ctc\Cerr \alpha_n\right)
\]
which by \eqref{eq:Theta} and $4 \beta^2 \left(\Cdec \Ctc\Cerr\right)^3 \leq C_2$ implies $s_{n+1} \leq C_2 \Theta \left(\alpha_{n+1}\right)$.\\
Now from $\vartheta\left(t\right) = \sqrt{t} \varphi\left(t\right)$ we find $b^2\left(\varphi \left(\vartheta^{-1}\left(\frac{\sqrt{a}}{b}\right)\right)\right)^2 = a/\vartheta^{-1}\left(\frac{\sqrt{a}}{b}\right)$ and hence by \eqref{eq:p9}
\begin{align*}
d_{n+1}^2 &\leq 4\beta^2 \left(\varphi \left(\vartheta^{-1}\left(\frac{\sqrt{s_{n+1}}}{2\beta}\right)\right)\right)^2 \\[0.1cm]
&\leq 4\beta^2 \left(\varphi \left(\vartheta^{-1}\left(\frac{\sqrt{C_2}}{2\beta}\vartheta\left(\alpha_{n+1}\right)\right)\right)\right)^2 \\[0.1cm]
&\leq 2\beta\sqrt{C_2}\varphi \left(\alpha_{n+1}\right)^2 \\[0.1cm]
&\leq C_1^2 \varphi \left(\alpha_{n+1}\right)^2
\end{align*}
where we used \eqref{eq:varphif}, $C_2 \geq 4 \beta^2$ due to $\Cdec\Ctc\Cerr\geq 1$ and $\sqrt{2\beta}\sqrt[4]{C_2}\leq C_1$.\\[0.1cm]
Therefore, we have proven that \eqref{eq:ind1} and \eqref{eq:ind2} hold for all $n \leq n_*$ (or in case of exact data for all $n \in \mathbb N$).
\end{proof}
With these two lemmas at hand we are able to complete the Proof of Theorem \ref{thm:cr}: Inserting \eqref{eq:stopr} into \eqref{eq:ind1} and \eqref{eq:ind2} we find using \eqref{eq:varphi2Theta}
\[
\mathcal D^{u^*}_{\mathcal R} \left(u_{n_*},u^\dagger\right) \leq C_1\varphi^2\left(\alpha_{n_*}\right)  = \mathcal O \left(\varphi^2\left(\Theta^{-1} \left(\err_{n_*}\right)\right)\right)
\]
and 
\[
\T{F\left(u_{n_*}\right)}{g^\dagger} \leq C_2\Theta \left(\alpha_{n_*}\right)=\mathcal O \left(\err_{n_*}\right).
\]

\section{A Lepski{\u\i}-type stopping rule and additive source conditions}\label{sec:lepskii}
In this section we will present a convergence rates result under the following variational source condition in additive form:
\setcounter{ass}{4}
\begin{assb}\label{ass:additive_sc}
There exists $u^*\in\partial \mathcal{R}(u^{\dagger})\subset \Xspace'$,
parameters $\beta_1\in\left[0,1/2\right)$, $\beta_2>0$ (later assumed to be sufficiently small), 
and a strictly concave, differentiable index function 
$\varphi$ satisfying $\varphi'\left(t\right) \nearrow \infty$ as $t \searrow 0$ such that
\begin{equation}\label{eq:vieadd}
\left<u^*, u^\dagger - u\right> \leq \beta_1\mathcal D_{\mathcal R}^{u^*} 
\left(u, u^\dagger\right) + \beta_2\varphi\left(\T{F\left(u\right)}{g^\dagger}\right)\qquad 
\mbox{for all }u \in \mathfrak B\,.
\end{equation}
\end{assb}
A special case of condition \eqref{eq:vieadd}, motivated by the \textit{benchmark condition} $u^* = F\left[u^\dagger\right]^* \omega$ was first introduced 
in \cite{hkps07} to prove convergence rates of Tikhonov-type regularization in Banach spaces (see also \cite{s08}). Flemming \cite{f10} uses them to prove convergence rates for nonlinear Tikhonov regularization \eqref{eq:nltikh} with general 
$\calS$ and $\mathcal R$. Bot \& Hofmann \cite{bh10} prove convergence rates for general $\varphi$ and introduce the use of Young's inequality which we will apply 
in the following. Finally, Hofmann \& Yamamoto \cite{hy10} prove equivalence in the Hilbert space case for $\varphi\left(t\right) = \sqrt{t}$ in \eqref{eq:sc} and \eqref{eq:vieadd} (with different $\varphi$, cf. \cite[Prop. 4.4]{hy10}) and almost equivalence for $\varphi\left(t\right) = t^\nu$ with $\nu < \frac12$ in \eqref{eq:sc} (again with different $\varphi$ in \eqref{eq:vieadd}, cf. \cite[Prop. 6.6 and Prop. 6.8]{hy10}) under a suitable nonlinearity condition.

Latest research results show that a classic Hilbert space source conditions \eqref{eq:sc}, which have natural interpretations in a number of important examples, relates to \eqref{eq:vieadd} in a way that one obtains order optimal rates (see \cite{f11diss}). Nevertheless, this can be seen much easier for multiplicative variational source conditions (see \eqref{eq:schspace}). 

The additive structure of the variational inequality will facilitate our proof and the result will give us the possibility to apply a Lepski{\u\i}-type stopping rule. We remark that for $\mathfrak{s}\neq 0$ in Assumption \ref{ass:SR}
it is not clear how to
formulate an implementable discrepancy principle.

\bigskip

Given $\varphi$ in \eqref{eq:vieadd}, we construct the following further
index functions as in \cite{bh10}, which will be used   
in our convergence theorem:
\begin{subequations}
\begin{align}
\label{eq:defi_psi}
\psi \left(t\right) &= \begin{cases} \frac{1}{\varphi' \left(\varphi^{-1}\left(t\right)\right)}&\text{if }t > 0,\\[0.1cm]
0 & \text{if }t = 0, \end{cases} = \begin{cases}\left(\varphi^{-1}\right)' \left(t\right)&\text{if }t > 0,\\[0.1cm]
0 & \text{if }t = 0, \end{cases} \displaybreak[1]\\[0.1cm]
\label{eq:defi_varPsi}
\varPsi\left(t\right) &= \int\limits_0^t \psi^{-1}\left(s\right) \,\mathrm d s, \qquad t \geq 0, \displaybreak[1]\\[0.1cm]
\label{eq:defi_Lambda}
\Lambda&= \inf\left\{g ~\big|~\sqrt{g}\text{ concave index function, } g\left(t\right)\geq \frac{\varPsi \left(t\right)}{t}\text{ for } t \geq 0\right\}.
\end{align}
\end{subequations}
The definition \eqref{eq:defi_Lambda} ensures that $\sqrt{\Lambda}$ is concave, which by \eqref{eq:parac} implies 
\begin{equation}\label{eq:decayLambda}
\left(\Lambda\left(\alpha_n\right)\right)^{\frac1q} \leq \Cdec^{\frac{2}{q}} \left(\Lambda\left(\alpha_{n-1}\right)\right)^{\frac1q}
\end{equation}
for all $q \geq 1$ and $n \in \mathbb N$. Since for linear problems $\sqrt{\varPsi\left(\alpha_n\right)/\alpha_n}$ is  
a bound on the approximation error (see \cite{bh10}) and since for
Tikhonov regularization the approximation error decays at most of the order
$O(\alpha_n)$, we expect that $t\mapsto \sqrt{\varPsi (t)/t}$ is ''asymptotically
concave'' in the sense that 
$\lim_{t\searrow 0}\Lambda(t)t/\varPsi(t) = 1$,
so we don't loose anything by replacing $\Psi(t)/t$ by $\Lambda(t)$.  
Indeed, it is easy to see that this is the case for logarithmic and 
H\"older type source conditions with $\nu \leq 1$, and in the latter 
case $t\mapsto \sqrt{\varPsi (t)/t}$ itself is concave everywhere.

\begin{lem}\label{lem:err_decomp}
Let Assumption~\ref{ass:F}, \ref{ass:SR}, \ref{ass:ex},  \ref{ass:nl}A or \ref{ass:nl}B and \ref{ass:additive_sc}B hold true and assume that there exists a uniform upper bound $\err_n\leq \err$ for the error terms $\err_n$ in Theorem \ref{thm:cr}. Then, with the notation \eqref{eq:abbrev}, the error of the iterates $u_n$ defined by \eqref{eqs:method} for $n \geq 1$ can be bounded by the
sum of an \emph{approximation error} bound $\Phi_{\rm app}(n)$, 
a \emph{propagated data noise error} bound $\Phi_{\rm noi}(n)$ and a \emph{nonlinearity error} bound $\Phi_{\rm nl}(n)$,
\begin{equation}\label{eq:err_dec}
d_n^2 \leq \Phi_{\rm nl}\left(n\right) +\Phi_{\rm app}\left(n\right)+ \Phi_{\rm noi}\left(n\right)
\end{equation}
where
\begin{align*}
\Phi_{\rm nl}\left(n\right) &:= 2\eta C_{\rm NL}\frac{s_{n-1}}{\alpha_{n-1}}, \\
\Phi_{\rm app} \left(n\right) &:= 2\beta_2 \Lambda \left(\alpha_{n-1}\right), \\
\Phi_{\rm noi} \left(n\right) &:= 2\frac{\err}{\alpha_{n-1}}.
\end{align*}
and $C_{\rm NL} := \max\left\{2 \Cerr, \Cerr + 1/\Cerr\right\}$. Moreover, if $\eta$ and $\beta_2$ are sufficiently small, the estimate 
\begin{equation}\label{eq:nlest2}
\Phi_{\rm nl}\left(n\right) \leq \gamma_{\rm nl} \left(\Phi_{\rm noi}\left(n\right) + \Phi_{\rm app} \left(n\right)\right)
\end{equation}
holds true with
\begin{align*}
\gamma_{\rm nl} &:=\max\left\{\frac{\Cdec^2 \bar \gamma}{1-\Cdec^2 \bar \gamma}, \frac{\Phi_{\rm nl}\left(1\right)}{\Phi_{\rm app}\left(1\right) + 
\Phi_{\rm noi}\left(1\right)}\right\}, \qquad
\bar \gamma := \frac{\eta \Cdec C_{\rm NL}}{\frac{1}{\Ctc\Cerr} -\beta_2}\,.
\end{align*}
\end{lem}

\begin{proof}
Similar to the proof of Lemma \ref{lem:2} the assumptions imply the iterative estimate
\[
\alpha_n \left(1-\beta_1\right)d_{n+1}^2 + \frac{1}{\Ctc\Cerr} s_{n+1}\leq \eta \left(\Cerr+\frac{1}{\Cerr}\right) s_n + \alpha_n\beta_2 \varphi\left(s_{n+1}\right)+\err
\]
for all $n \in \mathbb N$ in case of of \ref{ass:nl}B and
\[
\alpha_n \left(1- \beta_1\right) d_{n+1}^2 + \frac{1}{\Ctc\Cerr} s_{n+1} \leq 2 \eta \Cerr s_n + \alpha_n \beta_2 \varphi \left(s_{n+1}\right) + \err
\]
for all $n\in \mathbb N$ in case of \ref{ass:nl}A. Now Young's inequality $ab \leq \int_0^a \psi\left(t\right) \,\mathrm dt + \int_0^b\psi^{-1}\left(s\right) \,\mathrm d s$ (cf. \cite[Thm. 156]{hlp67})  with the index function $\psi$ 
defined in \eqref{eq:defi_psi} applied to the second-last term yields 
\[
\alpha_n\beta_2 \varphi\left(s_{n+1}\right) \leq \beta_2 s_{n+1} + \beta_2 \varPsi \left(\alpha_n\right)\,.
\]
This shows that 
\begin{equation}\label{eq:reeadd}
\alpha_n \left(1-\beta_1\right)d_{n+1}^2 + \left(\frac{1}{\Ctc\Cerr} -\beta_2\right) s_{n+1} \leq \eta C_{\rm NL} s_n  + \beta_2 \varPsi\left(\alpha_n\right) +\err
\end{equation}
for all $n \in \mathbb N$ both in case \ref{ass:nl}A and in case \ref{ass:nl}B. Together with $1/(1-\beta_1) \leq 2$ and $\frac{\varPsi \left(t\right)}{t} \leq \Lambda \left(t\right)$ this yields 
\[
d_{n+1}^2 \leq 2\eta C_{\rm NL} \frac{s_n}{\alpha_n}  + 2\beta_2 \Lambda\left(\alpha_n\right) +2\frac{\err}{\alpha_n}.
\]
for all $n \geq 0$ which is by definition \eqref{eq:err_dec}. 

From \eqref{eq:reeadd} we conclude that 
\[
s_{n+1} \leq \frac{\eta C_{\rm NL}}{\frac{1}{\Ctc\Cerr} - \beta_2} s_n + \frac{\beta_2}{\frac{1}{\Ctc\Cerr} - \beta_2}\varPsi\left(\alpha_n\right) + \frac{\err}{\frac{1}{\Ctc\Cerr} - \beta_2}.
\]
Now multiplying by $2\eta C_{\rm NL}/\alpha_{n+1}$ we find
\[
\Phi_{\rm nl}\left(n+2\right)\leq  \bar \gamma\Phi_{\rm nl}\left(n+1\right) + \bar \gamma\Phi_{\rm app}\left(n+1\right) + \bar \gamma\Phi_{\rm noi}\left(n+1\right)
\]
for all $n \in \mathbb N$.
Now we prove \eqref{eq:nlest2} by induction:
For $n = 1$ the assertion is true by the definition of $\gamma_{\rm nl}$. 
Now let \eqref{eq:nlest2} hold for some $n$. Then by the inequality above, 
the induction hypothesis, \eqref{eq:decayLambda}, 
and the monotonicity of $\Phi_{\rm noi}$ we find that 
\begin{align*}
\Phi_{\rm nl}\left(n+1\right)&\leq  \bar \gamma\Phi_{\rm nl}\left(n\right) + \bar \gamma\Phi_{\rm app}\left(n\right) + \bar \gamma\Phi_{\rm noi}\left(n\right) \\[0.1cm]
& \leq \bar \gamma \left(1+\gamma_{\rm nl}\right) \left(\Phi_{\rm app}\left(n\right) + \Phi_{\rm noi}\left(n\right)\right) \\[0.1cm]
& \leq \Cdec^2 \bar \gamma \left(1+\gamma_{\rm nl}\right) \left(\Phi_{\rm app}\left(n+1\right) + \Phi_{\rm noi}\left(n+1\right)\right) \, .
\end{align*}
The definition of $\gamma_{\rm nl}$ implies $\Cdec^2 \bar \gamma \left(1+\gamma_{\rm nl}\right) \leq \gamma_{\rm nl}$ and hence the assertion is shown.
\end{proof}

Lemma \ref{lem:err_decomp} allows us to apply the Lepski{\u\i} balancing
principle as developed in \cite{MP:03,bh05,mathe:06,bhm09} as a posteriori 
stopping rule. Since the balancing principle requires a metric on $\Xspace$ 
we assume that \eqref{eq:lower_bound_breg} holds true. As already mentioned, this is for example the case if $\Xspace$ is a $q$-convex Banach space and $\mathcal{R}(u)  = \|u\|^q$. 

Together with \eqref{eq:lower_bound_breg} and taking the $q$-th root it follows from Lemma \ref{lem:err_decomp} that
\[
\|u_{n}-u^{\dagger}\| \leq \Cbreg^{\frac1q}\left(\Phi_{\rm nl}\left(n\right)^{\frac1q} +
\Phi_{\rm app}\left(n\right)^{\frac1q}+\Phi_{\rm noi}\left(n\right)^{\frac1q}\right).
\]
Whereas $\Phi_{\rm app}$ and $\Phi_{\rm nl}$ are typically unknown, 
it is important to note that the error component $\Phi_{\rm noi}$ is known
if an error bound $\err$ is available. 
Therefore, the following Lepski{\u\i} balancing principle can be implemented:
\begin{subequations}\label{eqs:lepskij}
\begin{align}
\label{eq:defi_Nmax}
N_ {\rm max} &:= \min \left\{n \in \mathbb N ~\big|~ \Cbreg^{\frac1q}\Phi_{\rm noi}\left(n\right)^{\frac1q} \geq 1\right\}\\
\label{eq:lepskij}
n_{\rm bal} &:= \min \left\{n \in  \{1, \dots, N_{\rm max}\} ~\big|~ 
\forall m \geq n\; \left\Vert u_n - u_m\right\Vert \leq c \Phi_{\rm noi}^{\frac1q}\left(m\right) 
\right\} 
\end{align}
\end{subequations}
Moreover, it is important to note that $\Phi_{\rm noi}$ is increasing and $\Phi_{\rm app}$ is
decreasing. Therefore, the general theory developed in the references above 
can be applied, and we obtain the following convergence result:

\begin{thm}[{Convergence rates under Assumption~\ref{ass:additive_sc}B}]\label{thm:cradd} 
 Let the assumptions of Lemma \ref{lem:err_decomp} hold true and assume that 
 $\mathcal D_{\mathcal R}^{u^*} \left(u_0,u^\dagger\right)$ and 
$\S{F\left(u_0\right)}{g^\dagger}$ are sufficiently small. 
\begin{enumerate}
\item \emph{exact data:}\\
Then the iterates $(u_n)$ defined by \eqref{eqs:method} with exact data $\gdelta = g^{\dagger}$ fulfill
\begin{equation}\label{eq:rateAdditiveEx}
\mathcal D^{u^*}_{\mathcal R} \left(u_n,u^\dagger\right) = \mathcal O \left(\Lambda \left(\alpha_n\right)\right), \qquad n \to \infty.
\end{equation}
\item \emph{a priori stopping rule:}\\
For noisy data and the stopping rule
\[
n_* := \min\left\{n \in \mathbb N ~\big|~ \varPsi\left(\alpha_n\right) \leq \err \right\}
\]
with $\varPsi$ defined in \eqref{eq:defi_varPsi} we obtain the convergence rate
\begin{equation}\label{eq:conv_apriori}
\mathcal D^{u^*}_{\mathcal R} \left(u_{n_*},u^\dagger\right) = \mathcal O \left(\Lambda \left(\varPsi^{-1} \left(\err\right)\right)\right)\,,
 \qquad \err \to 0.
\end{equation}
\item \emph{Lepski{\u\i}-type stopping rule:}\\
Assume that \eqref{eq:lower_bound_breg} holds true.
Then the Lepski{\u\i} balancing principle  \eqref{eq:lepskij} with 
$c =\Cbreg^{\frac1q}4 \left(1+ \gamma_{\rm nl}\right)$ leads to the convergence rate
\[
\left\Vert u_{n_{\rm bal}}- u^\dagger\right\Vert^q = \mathcal O \left(\Lambda \left(\varPsi^{-1} \left(\err\right)\right)\right), \qquad \err \to 0.
\]
\end{enumerate}
\end{thm}

\begin{proof}
By \eqref{eq:err_dec} and \eqref{eq:nlest2} we find $d_n^2 \leq \left(1+\gamma_{\rm nl}\right) \left(\Phi_{\rm app}\left(n\right) + \Phi_{\rm noi} \left(n\right)\right)$ which implies part 1 and
\[
d_{n_*}^2 \leq  \left(1+\gamma_{\rm nl}\right) \left(2 \beta_2 \Lambda \left(\alpha_{n_*-1}\right) + 2 \frac{\err}{\alpha_{n_*-1}}\right).
\]
Using the definition of $n_*$ and \eqref{eq:decayLambda} we have
\[
\frac{\err}{\alpha_{n_*-1}} 
\leq \frac{\varPsi\left(\alpha_{n_*-1}\right)}{\alpha_{n_*-1}}
 \leq \Lambda\left(\alpha_{n_*-1}\right) \leq \Cdec^2\Lambda\left(\alpha_{n_*}\right)\,.
\]
Using the definition of $n_*$ again we obtain $\alpha_{n_*}\leq \varPsi^{-1}\left(\err\right)$.
Putting these estimates together yields \eqref{eq:conv_apriori}.

To prove part 3 assume that $\err$ is sufficiently small in the following. We use again $d_n^2 \leq \left(1+\gamma_{\rm nl}\right) \left(\Phi_{\rm app}\left(n\right) + \Phi_{\rm noi} \left(n\right)\right)$, which yields by \eqref{eq:lower_bound_breg} the estimate
\[
\left\Vert u_n - u^\dagger\right\Vert \leq \Cbreg^{\frac1q}\left(1+\gamma_{\rm nl}\right)^{\frac1q}\left(\Phi_{\rm app}\left(n\right)^{\frac1q} + \Phi_{\rm noi} \left(n\right)^{\frac1q}\right)
\]
for all $n \in \left\{1, ..., N_{\rm max}\right\}$. Define $\psi\left(j\right) := 2 \Cbreg^{\frac1q}\left(1+\gamma_{\rm nl}\right)^{\frac1q} \Phi_{\rm noi} \left(N_{\rm max} + 1 - j\right)$ and $\phi\left(j\right) := 2 \Cbreg^{\frac1q}\left(1+\gamma_{\rm nl}\right)^{\frac1q} \Phi_{\rm app} \left(N_{\rm max} + 1 - j\right)$ and note that $\phi\left(1\right) \leq \psi\left(1\right)$ if and only if $\Phi_{\rm app} \left(N_{\rm max}\right) \leq 1$. This is the case if $N_{\rm max}$ is sufficiently large which holds true for sufficiently small $\err$ as assumed. Thus by \eqref{eq:decayLambda} we can apply \cite[Cor. 1]{mathe:06} to gain
\[
\left\Vert u_{n_{\rm bal}} - u^\dagger\right\Vert \leq 6 \left(1+\gamma_{\rm nl}\right)^{\frac1q}\Cdec^{\frac2q}\Cbreg^{\frac1q} \min\limits_{n \leq N_{\rm max}} \left(\Phi_{\rm app}\left(n\right)^{\frac1q} + \Phi_{\rm noi}\left(n\right)^{\frac1q}\right).
\]
If we can show that $n_* \in \left\{1, ..., N_{\rm max}\right\}$ we obtain the assertion as in part 2. Since by definition $\alpha_{n_*-1} > \varPsi^{-1} \left(\err\right)$, we have
\[
\Phi_{\rm noi} \left(n_*\right) = 2 \frac{\err}{\alpha_{n_*-1}} < 2 \frac{\err}{\varPsi^{-1} \left(\err\right)} \leq 2 \Lambda \left(\varPsi^{-1} \left(\err\right)\right)
\]
and hence $n_* \leq N_{\rm max}$ if $\err$ is sufficiently small.
\end{proof}

\section{Relation to previous results}\label{sec:spec}
The most commonly used source conditions are H\"older-type and logarithmic source conditions, which correspond to 
\begin{subequations}
\begin{eqnarray}
\label{eq:hoelder_sc}
\varphi_{\nu} \left(t\right) &:=& t^{\nu},\qquad \nu \in \left(0,1/2\right],\\
\label{eq:logarithmic_sc}
\bar\varphi_p\left(t\right) &:=& \begin{cases}  
\left(-\ln \left(t\right)\right)^{-p}& \text{if } 0 < t \leq \exp \left(-p-1\right),\\
0 & \text{if }t = 0, 
\end{cases}
\qquad p>0,
\end{eqnarray}
\end{subequations}
respectively. For a number of inverse problems such source conditions
have been shown to be equivalent to natural smoothness
assumptions on the solution in terms of Sobolev space regularity
(see \cite{ehn96,h00}). 
We have restricted the range of H\"older indices to 
$\nu \in \left(0,1/2\right]$ since for $\nu>1/2$ 
the monotonicity assumption \eqref{eq:condvarphi} is violated. By computing the second derivative, one
can easily see that the functions $\bar\varphi_p$ are concave on the
interval $[0,\exp(-p-1)]$, and condition \eqref{eq:condvarphi} is trivial. 
If necessary, the functions $\bar\varphi_p$ can be extended to concave
functions on $[0,\infty)$ by suitable affine linear function on 
$(\exp(-p-1),\infty)$. 

We note the explicit form of the abstract error estimates \eqref{eqs:conv_in_XY} 
for these classes of source conditions as a corollary:
\begin{cor}[H\"older and logarithmic source conditions]
Suppose the assumptions of Theorem \ref{thm:cr} hold true.
\begin{enumerate}
\item If $\varphi$ in \eqref{eq:scgen} is of the form 
\eqref{eq:hoelder_sc} and 
$n_*  := \min\left\{n \in \mathbb N ~\big|~ \alpha_n \leq \tau \err_n^{\frac{1}{1+2\nu}}\right\}$ with $\tau\geq 1$ sufficiently large, 
then 
\begin{subequations}
\begin{align}
\label{eq:rate_Hoelder}
\mathcal D^{u^*}_{\mathcal R} \left(u_{n_*},u^\dagger\right) 
&= \mathcal O \left(\err_{n_*}^{\frac{2\nu}{1+2\nu}}\right).
\end{align}
\end{subequations}
\item If $\varphi=\bar\varphi_p$, 
$\bar n_*   := \min\left\{n \in \mathbb N ~\big|~ \alpha_n^2 \leq \tau \err_n\right\}$ 
and $\tau\geq 1$ sufficiently large, then 
\begin{subequations}
\begin{align}
\label{eq:rate_logarithmic}
\mathcal D^{u^*}_{\mathcal R} \left(u_{\bar n_*},u^\dagger\right) 
&= \mathcal O \left(\bar\varphi_{2p}\left(\err_{\bar n_*}\right)\right).
\end{align}
\end{subequations}
\end{enumerate}
\end{cor}

\begin{proof}
In the case of H\"older source conditions we already remarked that 
the conditions in Assumption~\ref{ass:sc}A are satisfied $\nu\in (0,1/2]$,
and we have $\Theta\left(t\right) = t^{1+2\nu}$, 
$\Theta^{-1}(\xi) = \xi^{1/(1+2\nu)}$. 

In the case of logarithmic source conditions we have 
$\Theta\left(t\right) = t \cdot \bar\varphi_{2p}\left(t\right).$
The function $\Theta^{-1}$ does not have an algebraic representation, but 
its asymptotic behavior at $0$ can be computed:
$
\Theta^{-1}\left(t\right) = \frac{t}{\bar\varphi_{2p}\left(t\right)}\left(1+o \left(1\right)\right)$ as $t \searrow 0.$
This implies that $\bar\varphi_{p}\left(\Theta^{-1}\left(t\right)\right) 
= \bar\varphi_p \left(t\right) \left(1 + o\left(1\right)\right)$ as $t\searrow 0$.
Note that the proposed stopping rule $\bar n_*$, which can be implemented 
without knowledge of the smoothness index $p$,  deviates from the stopping rule
\[
n_*:=\min\left\{n \in \mathbb N ~\big|~ \alpha_n \bar\varphi_{2p}(\alpha_n) 
\leq \tau \err_n\right\}
\]
proposed in Theorem \ref{thm:cr}. Asymptotically we
have $n_*>\bar n_*$, and hence  \eqref{eq:cr_ex_data} holds for $n=\bar n_*$. 
Therefore, we still get the optimal rates since 
\[
\mathcal D^{u^*}_{\mathcal R} \left(u_{\bar n_*}, u^\dagger\right) 
= \mathcal O \left(\bar\varphi_{2p} \left(\alpha_{\bar n_*}\right)\right)
= \mathcal O \left(\bar\varphi_{2p} \left(\sqrt{\tau\err_{\bar n_*}}\right)\right)
= \mathcal O \left(\bar\varphi_{2p} \left(\err_{\bar n_*}\right)\right)\,.
\]
\end{proof}

Recall from section \ref{sec:assumptions} that we can choose 
\[
\err \equiv \delta^r\quad \mbox{if}\quad 
\|\gdelta-g^{\dagger}\|_{\Yspace}\leq \delta
\quad \mbox{and}\quad \S{g_1}{g_2}=\|g_1-g_2\|_{\Yspace}^r, ~\calT =\calS
\]
with $r\in [1,\infty)$. In particular, 
if $\Xspace$ and $\Yspace$ are Hilbert spaces, $r=2$ and $\mathcal R = \left\Vert u-u_0\right\Vert^2$ for some $u_0 \in \Xspace$, then \eqref{eq:rate_Hoelder} 
and \eqref{eq:rate_logarithmic} translate into the rates
\begin{align*}
\|u_{n_*}-u\| &= \mathcal O  \left(\delta^{\frac{2\nu}{1+2\nu}}\right),\\
\|u_{n_*}-u\| &= \mathcal O  \left((-\ln \delta)^{-p}\right),
\end{align*}
respectively, for $\delta\to 0$ (see, e.g., \cite{kns08}), 
which are known to be optimal for linear inverse problems. 

It remains to discuss the relation of Assumption~\ref{ass:nl} to the standard tangential cone condition:
\begin{lem}[tangential cone condition]\label{lem:tan_cone}
Let $\S{g_1}{g_2}=\T{g_1}{g_2}=\|g_1-g_2\|_{\Yspace}^r$. If $F$ fulfills the tangential cone condition
\begin{equation}\label{eq:tan_cone}
\left\Vert F\left(u\right) +F'\left(u;v-u\right)- F\left(v\right) \right\Vert_{\Yspace} \leq \bar \eta \left\Vert F\left(u\right) - F\left(v\right) \right\Vert_{\Yspace} \qquad \text{for all }u,v \in \mathfrak B
\end{equation}
with $\bar \eta\geq 0$ sufficiently small, then Assumptions \ref{ass:nl}A and
\ref{ass:nl}B are satisfied.
\end{lem}

\begin{proof}
Using the inequality $\left(a+b\right)^r \leq 2^{r-1}\left(a^r + b^r\right)$, $a,b \geq0$ we find that
\begin{align*}
&\left\Vert F\left(u\right) + F'\left(u;v-u\right) - g\right\Vert_{\Yspace}^r \\[0.1cm]
\leq& \left(\left\Vert F\left(u\right) + F'\left(u;v-u\right) -F\left(v\right)\right\Vert_{\Yspace} +\left\Vert F\left(v\right) - g\right\Vert_{\Yspace}\right)^r \\[0.1cm]
\leq & 2^{r-1}\bar \eta ^r \left\Vert F\left(u\right) - F\left(v\right) \right\Vert_{\Yspace}^r+2^{r-1}\left\Vert F\left(v\right) - g\right\Vert_{\Yspace}^r\\[0.1cm]
\leq &2^{2r-2} \bar \eta ^r \left\Vert F\left(u\right) - g\right\Vert_{\Yspace}^r + \left(2^{r-1} + \bar \eta^r 2^{2r-2}\right) \left\Vert F\left(v\right) - g\right\Vert_{\Yspace}^r.
\end{align*}
Moreover, with $\left|a-b\right|^r \geq 2^{1-r} a^r - b^r$, $a,b \geq 0$ we get
\begin{align*}
&\left\Vert F\left(u\right) + F'\left(u;v-u\right) - g\right\Vert_{\Yspace}^r \\[0.1cm]
\geq& \left|\left\Vert F\left(v\right) - g\right\Vert_{\Yspace}-\left\Vert F\left(u\right) + F'\left(u;v-u\right) -F\left(v\right)\right\Vert_{\Yspace}\right|^r \\[0.1cm]
\geq & 2^{1-r}\left\Vert F\left(v\right) - g \right\Vert_{\Yspace}^r- \bar \eta ^r \left\Vert F\left(u\right) - F\left(v\right)\right\Vert_{\Yspace}^r\\[0.1cm]
\geq &2^{1-r}\left\Vert F\left(v\right) - g) \right\Vert_{\Yspace}^r- 2^{r-1} \bar \eta ^r \left\Vert F\left(u\right) - g\right\Vert_{\Yspace}^r - 2^{r-1}\bar \eta^r \left\Vert F\left(v\right)-g\right\Vert_{\Yspace}^r\\[0.1cm]
=& \left(2^{1-r} - 2^{r-1} \bar \eta ^r \right)\left\Vert F\left(v\right) - g) \right\Vert_{\Yspace}^r - 2^{r-1}\bar \eta^r \left\Vert F\left(u\right) - g\right\Vert_{\Yspace}^r
\end{align*}
for all $g \in \Yspace$. Hence, \eqref{eqs:tc} holds true with $\eta = 2^{2r-2}\bar \eta^r$ and 
\[
\Ctc = \max \left\{\frac{1}{2^{1-r} - 2^{r-1} \bar \eta ^r},2^{r-1} + \bar \eta^r 2^{2r-2} \right\} \geq 1
\]
if $\bar \eta$ is sufficiently small. 
\end{proof}

\section{Convergence analysis for Poisson data}\label{sec:kl}
In this section we discuss the application of our results to inverse
problems with Poisson data. We first describe a natural continuous setting 
involving Poisson processes (see e.g.\ \cite{ab06}). The relation to
the finite dimensional setting discussed in the introduction is
described at the end of this section. 

Recall that a Poisson process with intensity $g\in L^1(\manifold)$ on some submanifold $\manifold\subset\mathbb{R}^d$ 
can be described as a random finite set of points 
$\{x_1,\dots,x_N\}\subset \manifold$ 
written as random measure $G=\sum_{n=1}^N \delta_{x_n}$ such that the following conditions are satisfied: 
\begin{enumerate}
\item
For all measurable subsets $\manifold'\subset\manifold$ the number 
$G(\manifold')=\#\{n:x_n\in\manifold'\}$ is Poisson distributed with mean
$\int_{\manifold'}g\,\mathrm dx$. 
\item
For disjoint measurable subsets $\manifold_1',\dots,\manifold_m'\subset \manifold$
the random variables $G(\manifold'_1),\dots, G(\manifold'_m)$ are stochastically 
independent. 
\end{enumerate}
Actually, the first condition can be replaced by the weaker assumption that
$\EW G(\manifold')= \int_{\manifold'}g\,\mathrm dx$. 
In photonic imaging $g$ will describe the photon density on the measurement
manifold $\manifold$, and $x_1,\dots,x_N$ with denote the positions of 
the detected photons. 
For a Poisson process $G$ with intensity $g$ and a 
measurable function $\psi:\manifold\to \mathbb{R}$ the following equalities
hold true whenever the integrals on the right hand sides exist
(see \cite{kingman:93}):
\begin{align}\label{eqs:integrals_Poisson}
&\EW \int\limits_{\manifold} \psi \,\mathrm dG 
= \int\limits_{\manifold} \psi g^{\dagger}\, \mathrm dx \,,\qquad \qquad 
\Var \int\limits_{\manifold} \psi \,\mathrm dG 
= \int\limits_{\manifold} \psi^2 g^{\dagger}\, \mathrm dx
\end{align}
We also introduce an exposure time $t>0$. Our convergence results 
will describe reconstruction errors in the limit $t\to\infty$. 
Assume the data $\tilde{G}_t$ are drawn from a Poisson 
process with intensity $tg^{\dagger}$ and define $G_t:=\frac{1}{t}\tilde{G}_t$.
The negative log-likelihood functional is given by  
\begin{equation}\label{eq:defi_calS}
\S{g}{G_t}= \begin{cases}\int\limits_{\manifold} g \,\mathrm dx 
- \int\limits_{\manifold} \ln g\,\mathrm dG_t
= \int\limits_{\manifold} g \,\mathrm dx - \frac{1}{t} \sum_{n=1}^N \ln g(x_n)\,,
&g\geq 0\\
\infty\,,&\mbox{else.}
\end{cases}
\end{equation}
We set $\ln 0 :=-\infty$, so $\S{g}{G_t}=\infty$ if  $g(x_n)=0$ for some 
$n=1,\dots,N$. 
Using \eqref{eqs:integrals_Poisson} we obtain the following formulas 
for the mean and variance of $\S{g}{G_t}$ if the integrals on the
right hand side exist:
\begin{align}\label{eqs:EW_VAR_KL}
\EW \S{g}{G_t} &= \int\limits_{\manifold}\left[g-g^{\dagger}\ln g\right]
\,\mathrm dx\,,
\qquad 
\Var\;  \S{g}{G_t} = \frac{1}{t}\int\limits_{\manifold} (\ln g)^2g^{\dagger}\,\mathrm dx\,.
\end{align}
The term $\mathfrak{s}(g^{\dagger})= \EW \S{g^\dagger}{G_t} 
= \int_{\manifold} [g^{\dagger}-g^{\dagger}\ln g^{\dagger}]\,\mathrm dx$ 
with $0\ln 0:=0$ 
is finite if $g^{\dagger}\in L^1(\manifold)\cap L^{\infty}(\manifold)$ and
$g^{\dagger}\geq 0$ as assumed below (see e.g.\ \cite[Lemma 2.2]{t04}).
Abbreviating the set $\{x\in\manifold:g^{\dagger}(x)>0\}$ by 
 $\{g^{\dagger}>0\}$ we set
\begin{equation}\label{eq:KL_cont}
\T{g}{g^{\dagger}} := \KL{g}{g^{\dagger}}:= 
\begin{cases}
\int\limits_{\{g^{\dagger}>0\}} 
\left[g-g^{\dagger}-g^{\dagger}\ln\frac{g}{g^{\dagger}}\right]\,\mathrm dx\,,&
g\geq 0\\
\infty\,,&\mbox{else.}
\end{cases}
\end{equation}
It can be shown that the integral is well-defined, possibly taking the value
$+\infty$, i.e.\ the negative part of  $-g^{\dagger}\ln (g^{\dagger}/g)$ is 
integrable if $g,g^{\dagger}\in L^1(\manifold)$ 
and $g,g^{\dagger}\geq 0$ (see e.g.\ \cite[Lemma 2.2]{t04}). 
We find that Assumption 
\ref{ass:SR} holds true with $\Cerr=1$ and
\begin{equation}\label{eq:err_Poisson}
\err(g):= \begin{cases}
\left|\int_{\manifold} \ln(g) 
\left(\mathrm dG_t-g^{\dagger} \,\mathrm dx\right)\right|\,,&g\geq 0\\
0\,,& \mbox{else.}
\end{cases}
\end{equation}
This motivates the following assumption:
\begin{assp}\label{ass:P}
With the notation of Assumption \ref{ass:F} assume that
\begin{enumerate}
	\item $\manifold$ is a compact submanifold of $\mathbb{R}^d$, 
	$\Yspace:=L^1(\manifold)\cap C(\manifold)$  	with norm 
	$\|g\|_{\mathcal{Y}}:=\|g\|_{L^1}+\|g\|_{\infty}$
	and 
\[
F(u)\geq 0 \qquad \mbox{for all }u\in\mathfrak{B}.
\]
\item For a subset $\tilde{\Yspace}\subset \Yspace$ specified
later there exist constants $\rho_0,t_0> 0$ and a strictly monotonically decreasing function 
$\zeta: (\rho_0,\infty)\to [0,1]$ fulfilling $\lim_{\rho\to \infty} \zeta(\rho) =0$
such that 
\begin{equation}\label{eq:conc}
\prob\left(\sup_{g\in \tilde{\Yspace}} \left|\int\limits_{\manifold} 
\ln(g) \left(\mathrm dG_t- g^{\dagger}\,\mathrm dx\right)\right|
\geq \frac{\rho}{\sqrt{t}}\right) \leq \zeta(\rho)
\end{equation}
for all $\rho>\rho_0$ and all $t>t_0$. 
\end{enumerate}
\end{assp}
 
It remains to discuss the concentration inequality \eqref{eq:conc}. 
A general result of this type, which can be seen as an analog to Talagrand's
inequalities for empirical processes, has been shown by 
Reynaud-Bouret \cite[Corollary 2]{rb03}.  She proved that for a 
Poisson process $G$ with intensity $\overline{g}\in L^1(\manifold)$ and a countable
family of functions $\{f_n\}_{n\in\Nset}$ with values in $[-b,b]$ the
random variable $Z:=\sup_{n\in\Nset} \left|\int f_n\,
(\mathrm dG-\overline{g}\,\mathrm dx\,)\right|$
satisfy the concentration inequality
\begin{equation}\label{eq:conc_Bouret}
\prob\paren{Z\geq (1+\epsilon)\EW(Z) + \sqrt{12v_0\rho}+\kappa(\epsilon)b\rho}
\leq \exp(-\rho)
\end{equation}
for all $\rho,\epsilon>0$ with $v_0:= \sup_{n\in\Nset} \int f_n^2 \overline{g}
\,\mathrm dx$
and $\kappa(\epsilon)= 5/4 + 32/\epsilon$. We can apply this result
with $G = tG_t$ and $\overline{g} = tg^{\dagger}$ if $\tilde{\Yspace}$
is separable and $\|\ln(g)\|_{\infty}\leq b$ for all 
$g\in \tilde{\Yspace}$. Under additional regularity assumptions
(e.g.\ $\manifold$ Lipschitz domain and $\sup\{\|\ln(g)\|_{H^s}:
g\in\tilde{\Yspace} \}<\infty$ with $s>\dim(\manifold)/2$) 
it can be shown that $\EW(Z)\leq C/\sqrt{t}$ 
(see \cite[sec.\ 4.1]{werner:12}). This yields a concentration inequality
of the form \eqref{eq:conc} with $\zeta(\rho):= \exp(-c\rho)$ for some
$c>0$.

An essential restriction of Reynaud-Bouret's concentration inequality
in our context is the assumption $\|\ln(g)\|_{\infty}\leq b$
for all $g\in\tilde{\Yspace}$. This does not allow for zeros
of $F(u)$ even on sets of measure $0$ if $F(u)$ is continuous, which
is a very restrictive assumption. 
Therefore, we introduce the following shifted version of the Kullback-Leibler divergence \eqref{eq:defi_KL} involving an offset parameter $\offset\geq 0$ and a side-constraint $g\geq -\frac{\offset}{2}$: 
\begin{align}\label{eq:kle_def}
\T{g}{g^{\dagger}} := \begin{cases} \KL{g+\offset}{g^{\dagger}+\offset} & \text{if } g \geq -\frac{\offset}{2}\\[0.1cm]\infty & \text{otherwise.}\end{cases}
\end{align}
Note that \eqref{eq:KL_cont} and \eqref{eq:kle_def} coincide for $\offset=0$. 
Correspondingly, we choose 
\begin{align}\label{eq:Se}
\S{g}{G_t}:= \begin{cases} 
\int_{\manifold} \left[g -\offset\ln(g+\offset)\right]\,\mathrm dx -
\int_{\manifold} \ln(g+\offset) \mathrm dG_t &
 \text{if } 
g \geq-\frac{\offset}{2} ,\\[0.1cm] 
\infty & \text{else}\end{cases}
\end{align}
as data misfit functional in \eqref{eq:NMgen}. 
Setting $\mathfrak{s}(g^{\dagger}):= \int_{\manifold} 
[g^{\dagger} - (g^{\dagger}+\offset)\ln(g^{\dagger}+\offset)]\,\mathrm dx$,
Assumption \ref{ass:SR} is satisfied with 
\begin{align}\label{eq:err_e}
\err\paren{g} 
:= \begin{cases} 
\int_{\manifold} \ln\left(g+\offset\right) 
\paren{\mathrm dG_t-g^{\dagger}\,\mathrm dx}, 
&g\geq - \frac{\offset}{2}, \\[0.1cm] 
0 & \mbox{else.}
\end{cases}
\end{align}

\begin{rem}[Assumptions \ref{ass:sc}A and \ref{ass:additive_sc}B (source conditions)] Using the inequality
\[
\|g_1-g_2\|_{L^2}^2 \leq \left(\frac{4}{3}\|g_1\|_{L^{\infty}} + \frac{2}{3}\|g_2\|_{L^{\infty}}\right)
\KL{g_1}{g_2}
\]
(see \cite[Lemma 2.2 (a)]{bl91}), Assumption~\ref{ass:sc}A/B with $\T{g_2}{g_1}= \|g_1-g_2\|_{L^2}^2$ imply 
 Assumption~\ref{ass:sc}A/B with $\T{g_2}{g_1}= \KL{g_2}{g_1}$ if $F(\mathfrak{B})$ is bounded in $L^{\infty}(\manifold)$. 
However, Assumptions  \ref{ass:sc}A/B with $\T{g_2}{g_1}= \KL{g_2}{g_1}$ may be fulfilled with a better
 index function $\varphi$ if $F(u^{\dagger})$ is close to $0$ in parts of the domain. 
 \end{rem}

Before we state our convergence result, we introduce the smallest concave
function larger than the rate function in Theorem \ref{thm:cradd}:
\begin{equation}\label{eq:defi_barphi}
\hat \varphi:= \inf\left\{\tilde{\varphi} ~\big|~\tilde{\varphi}
\text{ concave index function, } \tilde{\varphi}\left(s\right)\geq \Lambda \left(\varPsi^{-1} \left(s\right)\right)\text{ for } s \geq 0\right\}.
\end{equation}
From the case of H\"older-type source conditions we expect that
$\hat \varphi$ will typically coincide with $\Lambda\circ\varPsi^{-1}$
at least in a neighborhood of $0$ (see e.g. \cite[Prop. 4.3]{hy10}).
\begin{cor}
Let the Assumptions \ref{ass:F}, \ref{ass:ex} and \ref{ass:additive_sc}B hold true. 
Moreover, assume that one of the following conditions is satisfied:
\begin{itemize}
\item
Assumptions \ref{ass:nl}A and $\mathcal{P}$ hold true with $\mathcal{S}$
and $\mathcal{T}$ given by \eqref{eq:defi_calS} and 
\eqref{eq:KL_cont} and 
$\tilde{\Yspace} = F(\mathfrak{B})$. 
\item Assumptions \ref{ass:nl}B and $\mathcal{P}$ hold true with
 $\mathcal{T}$ and $\mathcal{S}$ given by 
\eqref{eq:kle_def} and \eqref{eq:Se} and 
\begin{align*}
\tilde{\Yspace}:=&\{F(u)+\offset: u\in\mathfrak{B}\}\\
& \cup\,\left\{F(u)+F'(u;v-u)+\offset: u,v\in \mathfrak{B}, F(u)+F'(u;v-u)\geq -\frac{\offset}{2}\right\}.
\end{align*}
\end{itemize}
Suppose that $\beta_2$ is sufficiently small, $\mathfrak B$ is bounded and $\mathcal R$ is chosen such that \eqref{eq:lower_bound_breg} holds true, and Lepski{\u\i}'s balancing principle \eqref{eqs:lepskij} is applied 
with $c =\Cbreg^{\frac1q} 4 \left(1+\gamma_{\rm nl}\right)$ and $\err= \frac{\tau\zeta^{-1} \left(1/\sqrt{t}\right)}{\sqrt{t}}$ with a 
sufficiently large parameter $\tau$ (a lower will be given in the proof). 
Then we obtain the following convergence rate in expectation:
\begin{align}\label{eq:KL_rate_aposteriori}
\EW \left\Vert u_{n_{\rm bal}}- u^\dagger\right\Vert^q \leq \mathcal O \left(\hat \varphi\left(\frac{\zeta^{-1}  (1/\sqrt{t})}{\sqrt{t}}\right)\right)\,,\qquad t\to \infty.
\end{align}
\end{cor}

\begin{proof}
In the case of Assumption~\ref{ass:nl}A and $\offset=0$, we find that
Assumption \ref{ass:SR} holds true with $\err$ defined by 
\eqref{eq:err_Poisson}. Assumption $\mathcal{P}$ implies 
that the terms $\err_n$ defined by 
\eqref{eq:defi_errnA} in Theorem \ref{thm:cr} satisfy
\begin{align}\label{eq:pr_cor_kl}
\prob \left[\sup\limits_{n\in\Nset_0}\err_n \leq \frac{\tau\rho}{\sqrt{t}}\right]
\geq 1-\zeta(\rho)
\end{align}
for all $\rho>\rho_0$ and $t>t_0$ with 
$\tau:=1+2\eta C_{\rm tc}+ C_{\rm tc}$ due to $C_{\rm err} = 1$. 
To show the analogous
estimate in the case of Assumption~\ref{ass:nl}B, recall 
that Assumption \ref{ass:SR} holds true with $\err$ defined 
by \eqref{eq:err_e}.
From the variational characterization of $u_{n+1}$ it follows that
\begin{equation}\label{eq:errFlin1}
F\left(u_n\right) + F'\left(u_n;u_{n+1} - u_n\right) \geq -\frac{\offset}{2}
\end{equation}
Moreover, from Assumption~\ref{ass:nl}B we conclude that 
\begin{equation}\label{eq:errFlin2}
F\left(u_n\right) + F'\left(u_n;u^\dagger - u_n\right) \geq -\frac{\offset}{2}  
\end{equation}
This yields the inequality \eqref{eq:pr_cor_kl} with $\tau:=2$ 
also for $\err_n$ defined by \eqref{eq:defi_errnB} using 
Assumption $\mathcal{P}$. 

By virtue of \eqref{eq:pr_cor_kl} the sets 
$E_\rho := \left\{\sup_{n\in\Nset_0}\err_n \leq \frac{\tau\rho}{\sqrt{t}}\right\}$ have probability 
$\geq 1-\zeta\left(\rho\right)$ if $\rho>\rho_0$. Recall
that $\zeta$ is monotonically decreasing and define
$\rho \left(t\right) := \zeta^{-1} \left(1/\sqrt{t}\right)$ where we assume $t$ to be sufficiently large. 
We have
\begin{equation}\label{eq:proof_kl_prob}
\begin{aligned}
\EW \left\Vert u_{n_{\rm bal}} - u^\dagger\right\Vert^q \leq &2^q \left(\max\limits_{E_{\rho \left(t\right)}} \left\Vert u_{n_{\rm bal}} - u^\dagger\right\Vert^q \cdot \prob \left(E_{\rho\left(t\right)}\right) \right.\\[0.1cm]
&+ \left.\sup\limits_{u,v \in \mathfrak B} \left\Vert u-v\right\Vert ^q \prob\left(E_{\rho\left(t\right)}^C\right)\right).
\end{aligned}
\end{equation}
Now we can apply Theorem \ref{thm:cradd} to obtain the error bound
\[
\max\limits_{E_{\rho\left(t\right)}}\left\Vert u_{n_{\rm bal}} - u^\dagger\right\Vert^q 
\leq C_1 \hat\varphi\left(\err\right)
\leq C_1\tau \hat\varphi\left(\frac{\zeta^{-1}(1/\sqrt{t})}{\sqrt{t}}\right)
\]
with some constant $C_1 > 0$ for all sufficiently large $t$. In the last inequality we have used
the concavity of $\hat\varphi$. Plugging this into 
\eqref{eq:proof_kl_prob} yields
\[
\EW \left\Vert u_{n_{\rm bal}} - u^\dagger\right\Vert^q \leq 2^q 
\left(C_1\tau  \hat \varphi \left(\frac{\zeta^{-1}(1/\sqrt{t})}{\sqrt{t}}\right) + \frac{1}{\sqrt{t}}\sup\limits_{u,v \in \mathfrak B} \left\Vert u-v\right\Vert ^q \right).
\]
Since $\hat \varphi$ is concave, there exists $C_2 > 0$ such that $s \leq C_2 \hat \varphi\left(s\right)$ for all sufficiently small $s > 0$. 
Moreover, $\frac{1}{\sqrt{t}}$ in the second term is bounded by $\frac{1}{\rho_0}\frac{\zeta^{-1}(1/\sqrt{t})}{\sqrt{t}}$, and thus we obtain the assertion \eqref{eq:KL_rate_aposteriori}. 
\end{proof}

If $\zeta \left(\rho\right) = \exp\left(-c \rho\right)$ for some $c > 0$ 
as discussed above, then our convergence rates result 
\eqref{eq:KL_rate_aposteriori} means that we have to pay a logarithmic 
factor for adaptation to unknown smoothness by the Lepski{\u\i} principle.
It is known (see \cite{tsybakov:00}) that in some cases such a logarithmic
factor is inevitable. 

The most important issue is the verification of Assumption $\mathcal{P}$. 
In case of Assumption \ref{ass:nl}A this follows from the results
discussed above only under the restrictive assumption that
$F(u)$ is uniformly bounded away from $0$ for all $u\in\mathfrak{B}$. 
On the other hand for the case of Assumption \ref{ass:nl}B we 
find that Assumption $\mathcal{P}$ is satisfied under the mild condition 
\[
\sup_{u,v\in\mathfrak{B}}\|F(u)+F'(u,v-u)\|_{H^s}<\infty\,.
\]

{\bf Binning.}
Let us discuss the relation between the discrete data model discussed
in the introduction and the continuous model above. 
Consider a decomposition of the measurement manifold $\manifold$ into
$J$ measurable disjoint subdomains (bins) of positive 
measure $|\manifold_j|>0$:
\[
\manifold = \bigcup_{j=1}^J \manifold_j
\]
In practice each $\manifold_j$ may correspond to a detector counting the
number of photons in $\manifold_j$, so the measured data are 
\[
\underline{g}^{\rm obs}_j= tG_t(\manifold_j) = \#\{n\,\vert\, x_n\in\manifold_j\}\,,
\qquad j=1,\dots,J\,.
\]
Consider the linear operator $S_J:L^1(\manifold)\to \mathbb{R}^J$, 
$(S_Jg)_j:=\int_{\manifold_j}g\,\mathrm dx$ and the mapping 
$S_J^*\underline{g} := \sum_{j=1}^J |\manifold_j|^{-1}\underline{g}_j
{\bf 1}_{\manifold_j}$, which is adjoint
to $S_J$ with respect to the $L^2(\manifold)$ inner product and the 
inner product $\langle\underline{g},\underline{h}\rangle 
:= \sum_{j=1}^J|\manifold_j|^{-1}\underline{g}_j\underline{h}_j$. 
$P_J:=S_J^*S_J$ is the $L^2$-orthogonal projection onto the subspace
of functions, which are constant on each $\manifold_j$. 
$S_J$ can naturally be extended to measures such that 
$(S_J(G_t))_j = G_t(\manifold_j) =\frac{1}{t}\#\{n:x_n\in \manifold_j\}$. 
For distinction we denote the right hand sides of 
eqs.\ \eqref{eq:SforPoisson} and \eqref{eq:defi_KL} by $\underline{\mathcal{S}}_J$ and $\underline{\mathbb{KL}}_J$, 
and define $\mathcal{S}_{\infty}$ and $\mathbb{KL}_{\infty}$
by \eqref{eq:defi_calS} and \eqref{eq:KL_cont}. Then
\[
\underline{\mathcal{S}}_J\left(\underline{g}^{\rm obs};\underline{g}\right)
= \mathcal{S}_{\infty}\left(S_J^*\underline{g}^{\rm obs};
S_J^*\underline{g}\right)
\quad \mbox{and}\quad
\underline{\mathbb{KL}}_J\left(\underline{g}^{\dagger};\underline{g}\right)
= \mathbb{KL}_{\infty}\left(S_J^*\underline{g}^{\dagger};
S_J^*\underline{g}\right).
\]
The discrete data model above can be treated in the framework of our 
analysis by choosing 
\[
\S{g}{g^{\rm obs}} := 
\underline{\mathcal{S}}_{J}\left(\frac{1}{t}g^{\rm obs};
S_J g\right)\,,
\]
$\mathfrak{s}(g^{\dagger}):=\underline{\mathcal{S}}_{J}
\left(S_Jg^{\dagger};S_Jg^{\dagger}\right)$, and 
$\mathcal{T}:=\mathbb{KL}_{\infty}$. Then Assumption \ref{ass:SR} holds 
true with
\begin{equation}\label{eq:err_binning}
\begin{aligned}
\err(g):=& \left| \sum_{j=1}^J\ln((S_Jg)_j) 
\left(\frac{1}{t}\underline{g}^{\rm obs}_j-(S_Jg^{\dagger})_j\right)\right|\\
&+ \left|\mathbb{KL}_{\infty}\left(g^{\dagger};g\right)
- \mathbb{KL}_{\infty}\left(P_Jg^{\dagger};P_Jg\right)
\right|
\end{aligned}
\end{equation}
if $S_Jg\geq 0$, $\{j:(S_Jg)_j=0, (Sg^{\dagger})_j+\underline{g}^{\rm obs}_j>0\}=\emptyset$ and $\err(g):=\infty$ else. 
To achieve convergence, the binning has to be refined as $t\to\infty$.
The binning should be chosen such that the second term on the
right hand side of \eqref{eq:err_binning} (the discretization error) 
is dominated by the first term (the stochastic error) such that the
reconstruction error is determined by the number of observed photons 
rather than discretization effects. 

\section{applications and computed examples}\label{sec:applications}

{\bf Solution of the convex subproblems.}
We first describe a simple strategy 
to minimize the convex functional \eqref{eq:NMgen} with $\calS$  
as defined in \eqref{eq:Se} in each Newton step. For the moment we neglect the side condition $g\geq -\offset/2$ 
in \eqref{eq:Se}. For simplicity we further assume
that $\mathcal{R}$ is quadratic, e.g.\  $\mathcal{R}(u) = \|u-u_0\|^2$. 
We approximate $\S{g+h}{\gdelta}$ by the second order Taylor expansion
\[
\calS^{(2)}[\gdelta;g](h) := 
\S{g}{\gdelta} +
\int\limits_{\manifold}\left[\left(1-\frac{g^{\rm obs}+\offset}{g+\offset}\right) h 
+ \frac{1}{2}\frac{g^{\rm obs}+\offset}{\left(g+\offset\right)^2}h^2 \right]\,\mathrm d x
\]
and define an inner iteration 
\begin{equation}\label{eq:inner_it}
h_{n,l} := \argmin_h 
\left[ \calS^{(2)}\Big[\gdelta;F(u_n)+F'[u_n](u_{n,l}-u_n);\Big](h) 
+ \alpha_n \mathcal{R}(u_{n,l}+h)\right]
\end{equation}
for $l=0,1,\dots$ 
with $u_{n,0}:=u_n$ and $u_{n,l+1}:= u_{n,l} + s_{n,l} h_{n,l}$. 
Here the step-length parameter $s_{n,l}$ is chosen as the largest $s\in [0,1]$ for which 
$ s F'[u_n]\geq -\eta \offset- F(u_n)$ with a tuning parameter $\eta\in [0,1)$ (typically $\eta=0.9$). 
This choice of $s_{n,l}$ ensures that $F(u_n) + F'[u_n](u_{n,l+1}-u_n)\geq -\eta \offset$, i.e.\
\eqref{eq:inner_it} is a reasonable approximation to \eqref{eq:NMgen}, and $\eta=1/2$ ensures
that $u_{n,l+1}$ satisfies the side condition in \eqref{eq:Se}. It follows from the first order
optimality conditions, which are necessary and sufficient due to strict convexity here, that 
$u_{n,l}=u_{n,l+1}$ is the exact solution $u_{n+1}$ of \eqref{eq:NMgen} if $h_{n,l}=0$. 
Therefore, we stop the inner iteration if $\|h_{n,l}\|/\|h_{n,0}\|$
is sufficiently small. We also stop the inner iteration if $s_{n,l}$ is $0$ or too small. 

Simplifying and omitting terms independent of $h$ we can write \eqref{eq:inner_it} as a least squares
problem 
\begin{align}\label{eq:method_impl}
\begin{aligned}
h_{n,l} = \argmin\limits_{h}&\Bigg[\int\limits_{\manifold}  \frac{1}{2}\left( 
\frac{\sqrt{g^{\rm obs}+\offset}}{g_{n,l}+\offset} F'[u_n]h + \frac{g_{n,l} - g^{\rm obs}}{\sqrt{g^{\rm obs}+\offset}}\right)^2\,\mathrm d x\\[0.1cm]
&+ \alpha_n \mathcal R \left(u_{n,l}+h\right)\Bigg]
\end{aligned}
\end{align}
with $g_{n,l}:= F(u_n) + F'[u_n](u_{n,l}-u_n)$.
\eqref{eq:method_impl} is solved by the CG method applied to the normal equation.

In the examples below we observed fast convergence of the inner iteration \eqref{eq:inner_it}. In the phase retrieval problem we had problems with the
convergence of the CG iteration when $\alpha_n$ becomes too small. 
If the offset parameter $\offset$ becomes too small or if $\offset=0$ convergence 
deteriorates in general. This is not surprising since
the iteration \eqref{eq:inner_it} cannot be expected to converge to the exact solution 
$u_{n+1}$ of \eqref{eq:NMgen} if the side condition 
$F(u_n)+F'(u_n;u_{n+1}-u_n)\geq -\offset/2$  is active at $u_{n+1}$.
The design of efficient algorithms for this case 
will be addressed in future research.

\medskip

{\bf An inverse obstacle scattering problem without phase information.} 
The scattering of polarized, transverse magnetic (TM) time harmonic 
electromagnetic waves by 
a perfect cylindrical conductor with smooth cross section $D\subset\mathbb{R}^2$ 
is described by the equations
\begin{subequations}
\begin{align}
&\Delta u + k^2 u = 0,&& \mbox{in } \mathbb{R}^2\setminus D,\\
&\frac{\partial u}{\partial n}= 0,&& \mbox{on } \partial D,\\
\label{eq:SRC}
&\lim_{r\to \infty} \sqrt{r}\left(\frac{u_s}{r} -\textup{i}k u_s\right) =0,&&
\mbox{where } r:=|x|, u_s := u-u_i\,.
\end{align}
\end{subequations}
Here $D$ is compact, $\mathbb{R}^2\setminus D$ is connected, $n$ is the outer
normal vector on $\partial D$, and $u_i=\exp(\textup{i}k x\cdot d)$
is a plane incident wave with direction $d\in \{x\in \mathbb{R}^2:|x|=1\}$.
This is a classical obstacle scattering problems, and we refer
to the monograph \cite{CK:97} for further details and references. 
The Sommerfeld radiation condition \eqref{eq:SRC} implies the asymptotic
behavior
\[
u_s(x) = \frac{\exp(\textup{i}k|x|)}{\sqrt{|x|}}\left(u_{\infty}\left(\frac{x}{\left|x\right|}\right) + \mathcal O\left(\frac{1}{\left|x\right|}\right)\right)
\]
as $|x|\to\infty$, and $u_{\infty}$ is called
the \emph{far field pattern} or \emph{scattering amplitude} of $u_s$. 

\begin{figure}[!thb]
\setlength\fheight{4.9cm} \setlength\fwidth{4.9cm}
\centering
\subfigure[true obstacle and total field for an incident wave from ``South West'']{
\label{subfig:scattering_wave}
%
%
\begin{tikzpicture}

\begin{axis}[%
view={0}{90},
name=plot1,
scale only axis,
width=\fwidth,
height=\fheight,
xmin=-2.50716, xmax=2.50716,
ymin=-2.50716, ymax=2.50716,
axis on top]
\addplot graphics [xmin=-2.507163e+00, xmax=2.507163e+00, ymin=-2.507163e+00, ymax=2.507163e+00] {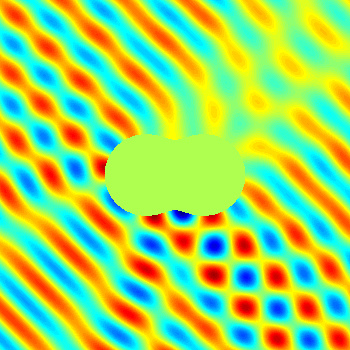};
\addplot [
color=black,
solid,
line width=2.0pt
]
coordinates{ (1,0) (0.999473,0.0245357) (0.997893,0.0490234) (0.995264,0.0734151) (0.991593,0.0976634) (0.986887,0.121721) (0.981158,0.145541) (0.974419,0.169078) (0.966686,0.192286) (0.957977,0.215121) (0.948312,0.23754) (0.937714,0.2595) (0.926208,0.280962) (0.913821,0.301886) (0.90058,0.322233) (0.886518,0.341968) (0.871667,0.361056) (0.856061,0.379466) (0.839736,0.397166) (0.82273,0.414128) (0.805083,0.430326) (0.786834,0.445736) (0.768026,0.460337) (0.748701,0.47411) (0.728904,0.487038) (0.708678,0.499108) (0.688071,0.510308) (0.667128,0.520632) (0.645896,0.530073) (0.624421,0.53863) (0.602752,0.546303) (0.580935,0.553096) (0.559017,0.559017) (0.537045,0.564076) (0.515066,0.568288) (0.493124,0.571668) (0.471265,0.574239) (0.449533,0.576023) (0.427969,0.577049) (0.406614,0.577348) (0.385509,0.576955) (0.364689,0.575907) (0.34419,0.574247) (0.324045,0.572019) (0.304283,0.569274) (0.284932,0.566062) (0.266014,0.56244) (0.247551,0.558466) (0.229557,0.554201) (0.212046,0.549709) (0.195024,0.545057) (0.178495,0.540312) (0.162455,0.535544) (0.146898,0.530823) (0.131811,0.526218) (0.117174,0.521799) (0.102964,0.517633) (0.0891499,0.513783) (0.0756975,0.510311) (0.0625659,0.507271) (0.0497098,0.504712) (0.0370797,0.502677) (0.0246223,0.501198) (0.0122817,0.500301) (3.06162e-17,0.5) (-0.0122817,0.500301) (-0.0246223,0.501198) (-0.0370797,0.502677) (-0.0497098,0.504712) (-0.0625659,0.507271) (-0.0756975,0.510311) (-0.0891499,0.513783) (-0.102964,0.517633) (-0.117174,0.521799) (-0.131811,0.526218) (-0.146898,0.530823) (-0.162455,0.535544) (-0.178495,0.540312) (-0.195024,0.545057) (-0.212046,0.549709) (-0.229557,0.554201) (-0.247551,0.558466) (-0.266014,0.56244) (-0.284932,0.566062) (-0.304283,0.569274) (-0.324045,0.572019) (-0.34419,0.574247) (-0.364689,0.575907) (-0.385509,0.576955) (-0.406614,0.577348) (-0.427969,0.577049) (-0.449533,0.576023) (-0.471265,0.574239) (-0.493124,0.571668) (-0.515066,0.568288) (-0.537045,0.564076) (-0.559017,0.559017) (-0.580935,0.553096) (-0.602752,0.546303) (-0.624421,0.53863) (-0.645896,0.530073) (-0.667128,0.520632) (-0.688071,0.510308) (-0.708678,0.499108) (-0.728904,0.487038) (-0.748701,0.47411) (-0.768026,0.460337) (-0.786834,0.445736) (-0.805083,0.430326) (-0.82273,0.414128) (-0.839736,0.397166) (-0.856061,0.379466) (-0.871667,0.361056) (-0.886518,0.341968) (-0.90058,0.322233) (-0.913821,0.301886) (-0.926208,0.280962) (-0.937714,0.2595) (-0.948312,0.23754) (-0.957977,0.215121) (-0.966686,0.192286) (-0.974419,0.169078) (-0.981158,0.145541) (-0.986887,0.121721) (-0.991593,0.0976634) (-0.995264,0.0734151) (-0.997893,0.0490234) (-0.999473,0.0245357) (-1,1.22465e-16) (-0.999473,-0.0245357) (-0.997893,-0.0490234) (-0.995264,-0.0734151) (-0.991593,-0.0976634) (-0.986887,-0.121721) (-0.981158,-0.145541) (-0.974419,-0.169078) (-0.966686,-0.192286) (-0.957977,-0.215121) (-0.948312,-0.23754) (-0.937714,-0.2595) (-0.926208,-0.280962) (-0.913821,-0.301886) (-0.90058,-0.322233) (-0.886518,-0.341968) (-0.871667,-0.361056) (-0.856061,-0.379466) (-0.839736,-0.397166) (-0.82273,-0.414128) (-0.805083,-0.430326) (-0.786834,-0.445736) (-0.768026,-0.460337) (-0.748701,-0.47411) (-0.728904,-0.487038) (-0.708678,-0.499108) (-0.688071,-0.510308) (-0.667128,-0.520632) (-0.645896,-0.530073) (-0.624421,-0.53863) (-0.602752,-0.546303) (-0.580935,-0.553096) (-0.559017,-0.559017) (-0.537045,-0.564076) (-0.515066,-0.568288) (-0.493124,-0.571668) (-0.471265,-0.574239) (-0.449533,-0.576023) (-0.427969,-0.577049) (-0.406614,-0.577348) (-0.385509,-0.576955) (-0.364689,-0.575907) (-0.34419,-0.574247) (-0.324045,-0.572019) (-0.304283,-0.569274) (-0.284932,-0.566062) (-0.266014,-0.56244) (-0.247551,-0.558466) (-0.229557,-0.554201) (-0.212046,-0.549709) (-0.195024,-0.545057) (-0.178495,-0.540312) (-0.162455,-0.535544) (-0.146898,-0.530823) (-0.131811,-0.526218) (-0.117174,-0.521799) (-0.102964,-0.517633) (-0.0891499,-0.513783) (-0.0756975,-0.510311) (-0.0625659,-0.507271) (-0.0497098,-0.504712) (-0.0370797,-0.502677) (-0.0246223,-0.501198) (-0.0122817,-0.500301) (-9.18485e-17,-0.5) (0.0122817,-0.500301) (0.0246223,-0.501198) (0.0370797,-0.502677) (0.0497098,-0.504712) (0.0625659,-0.507271) (0.0756975,-0.510311) (0.0891499,-0.513783) (0.102964,-0.517633) (0.117174,-0.521799) (0.131811,-0.526218) (0.146898,-0.530823) (0.162455,-0.535544) (0.178495,-0.540312) (0.195024,-0.545057) (0.212046,-0.549709) (0.229557,-0.554201) (0.247551,-0.558466) (0.266014,-0.56244) (0.284932,-0.566062) (0.304283,-0.569274) (0.324045,-0.572019) (0.34419,-0.574247) (0.364689,-0.575907) (0.385509,-0.576955) (0.406614,-0.577348) (0.427969,-0.577049) (0.449533,-0.576023) (0.471265,-0.574239) (0.493124,-0.571668) (0.515066,-0.568288) (0.537045,-0.564076) (0.559017,-0.559017) (0.580935,-0.553096) (0.602752,-0.546303) (0.624421,-0.53863) (0.645896,-0.530073) (0.667128,-0.520632) (0.688071,-0.510308) (0.708678,-0.499108) (0.728904,-0.487038) (0.748701,-0.47411) (0.768026,-0.460337) (0.786834,-0.445736) (0.805083,-0.430326) (0.82273,-0.414128) (0.839736,-0.397166) (0.856061,-0.379466) (0.871667,-0.361056) (0.886518,-0.341968) (0.90058,-0.322233) (0.913821,-0.301886) (0.926208,-0.280962) (0.937714,-0.2595) (0.948312,-0.23754) (0.957977,-0.215121) (0.966686,-0.192286) (0.974419,-0.169078) (0.981158,-0.145541) (0.986887,-0.121721) (0.991593,-0.0976634) (0.995264,-0.0734151) (0.997893,-0.0490234) (0.999473,-0.0245357) (1,0)
};

\end{axis}

\end{tikzpicture}
}
\setlength\fheight{2.18cm}
\subfigure[$t|u_{\infty}|^2=tF(q^{\dagger})$ for both waves (red line) and corresponding
count data $\gdelta$ (blue crosses)]{
\label{subfig:scattering_data}
\input{scattering_data.tikz}
}\\
\setlength\fheight{4.9cm}
\subfigure[results for $\calS$ as in \eqref{eq:Se}. blue: best, green: median,
black: initial guess]{
\label{subfig:scattering_kl}
\input{scattering_kl.tikz}
}
\subfigure[results for $\S{g_1}{g_2}=\|g_1-g_2\|_{L^2}^2$. blue: best, green: median,
black: initial guess]{
\label{subfig:scattering_kl_high}
\input{scattering_l2.tikz}
}
\caption{Numerical results for the inverse obstacle scattering problem \eqref{eq:scat}. Panels c) and
d) show best and median reconstruction from 100 experiments with $t=1000$ 
expected counts. See also Table \ref{tab:scattering}.}
\label{fig:scattering}
\end{figure}

We consider the inverse problem  to recover the shape of the obstacle $D$ 
from photon counts of the scattered electromagnetic field far away from the obstacle.
Since the photon density is proportional to the squared absolute value of the electric
field, we have no immediate access to the phase of the electromagnetic field. 
Since at large distances the photon density is approximately proportional to $|u_{\infty}|^2$, 
our inverse problem is described by the operator equation
\begin{equation}\label{eq:scat}
F(\partial D) = |u_{\infty}|^2\,.
\end{equation}
A similar problem is studied with different methods and noise models 
by Ivanyshyn \& Kress \cite{IK:10}. 
Recall that $|u_{\infty}|$ is invariant under translations of $\partial D$. 
Therefore, it is only possible to recover the shape, but not the location of 
$D$. For plottings we always shift the center of gravity of $\partial D$ to the origin.
We assume that $D$ is star-shaped and represent $\partial D$ by a periodic
function $q$ such that $\partial D = \{q(t)(\cos t,\sin t)^{\top}:t\in [0,2\pi]\}$.
For details on the implementation of $F$, its derivative and adjoint we refer
to \cite{hohage:98} where the mapping $q\mapsto u_{\infty}$ is considered
as forward operator. Even in this situation where the phase of $u_\infty$
is given in addition to its modulus, it has been shown in \cite{hohage:98}
that for Sobolev-type smoothness assumptions at most logarithmic rates
of convergence can be expected. 

\begin{table}[!htb]
\begin{center}
\begin{tabular}{cc|ccc|}
$t$ &$\S{g}{g^{\rm obs}}$ & $N$ &  
$\sqrt{\EW \|q_N\!-\!q^{\dagger}\|_{L^2}^2}$ & $\sqrt{\Var \|q_N\!-\!q^{\dagger}\|_{L^2}}$\\ [1ex]
\hline\hline
&$\|g-g^{\rm obs}\|_{L^2}^2$ & 7 & 0.124 & 0.033 \\
&$\phi^2 \left(g;\max\left\{\gdelta,0.2\right\}\right)$ & 2 & 0.122 & 0.018 \\
\raisebox{1.5ex}[-1.5ex]{$100$}&$\mathcal{S}$ in eq.~\eqref{eq:Se}& 3  & 0.091 & 0.025\\ [0.5ex]
\hline
&$\|g-g^{\rm obs}\|_{L^2}^2$ & 9 & 0.106 & 0.014 \\
&$\phi^2 \left(g;\max\left\{\gdelta,0.2\right\}\right)$ & 7 & 0.091 & 0.012 \\
\raisebox{1.5ex}[-1.5ex]{$1000$}&$\mathcal{S}$ in eq.~\eqref{eq:Se}& 5  & 0.070 & 0.017\\ [0.5ex]
\hline
&$\|g-g^{\rm obs}\|_{L^2}^2$ & 9 & 0.105 & 0.004 \\
&$\phi^2 \left(g;\max\left\{\gdelta,0.2\right\}\right)$ & 23 & 0.076 & 0.048 \\
\raisebox{1.5ex}[-1.5ex]{$10000$}&$\mathcal{S}$ in eq.~\eqref{eq:Se}& 5  & 0.050 & 0.005 \\ [0.5ex]
\hline
\end{tabular}
\end{center}

\caption{$L^2$-error statistics for the inverse obstacle scattering problem \eqref{eq:scat}. The log-likelihood functional \eqref{eq:Se} is compared to the 
standard $L^2$ and Pearson's $\phi^2$ distance (cf. \eqref{eq:pearson}) for different values of the expected total number
of counts $t$ with 100 experiments for each set of parameters. 
The error of the initial guess is $\|q_0\!-\!q^{\dagger}\|_{L^2}= 0.288$. 
All parameters as in Figure \ref{fig:scattering}.
}
\label{tab:scattering}
\end{table}
As a test example we choose the obstacle shown in Figure \ref{fig:scattering}
described by $q^{\dagger}(t) = \frac{1}{2}\sqrt{3\cos^2t+1}$ with two incident 
waves from ``South West'' and from ``East'' with wave number 
$k=10$ as shown in Figure \ref{fig:scattering}. We used $J=200$ equidistant
bins. 
The initial guess for the Newton iteration is the unit circle
described by $q_0\equiv 1$, and we choose the Sobolev norm 
$\mathcal R\left(q\right) = \left\Vert q-q_0\right\Vert_{H^s}^2$ with
$s=1.6$ as penalty functional. The regularization parameters are chosen
as $\alpha_n = 0.5\cdot(2/3)^n$. Moreover, we choose an initial offset parameter
$\offset=0.002$, which is reduced by $\frac45$ in each iteration step. The inner 
iteration \eqref{eq:inner_it} is stopped when $\|h_{n,l}\|/\|h_{n,0}\|\leq 0.1$, 
which was usually the case after about 3 iterations (or about 5 iterations for
$\|h_{n,l}\|/\|h_{n,0}\|\leq 0.01$). 

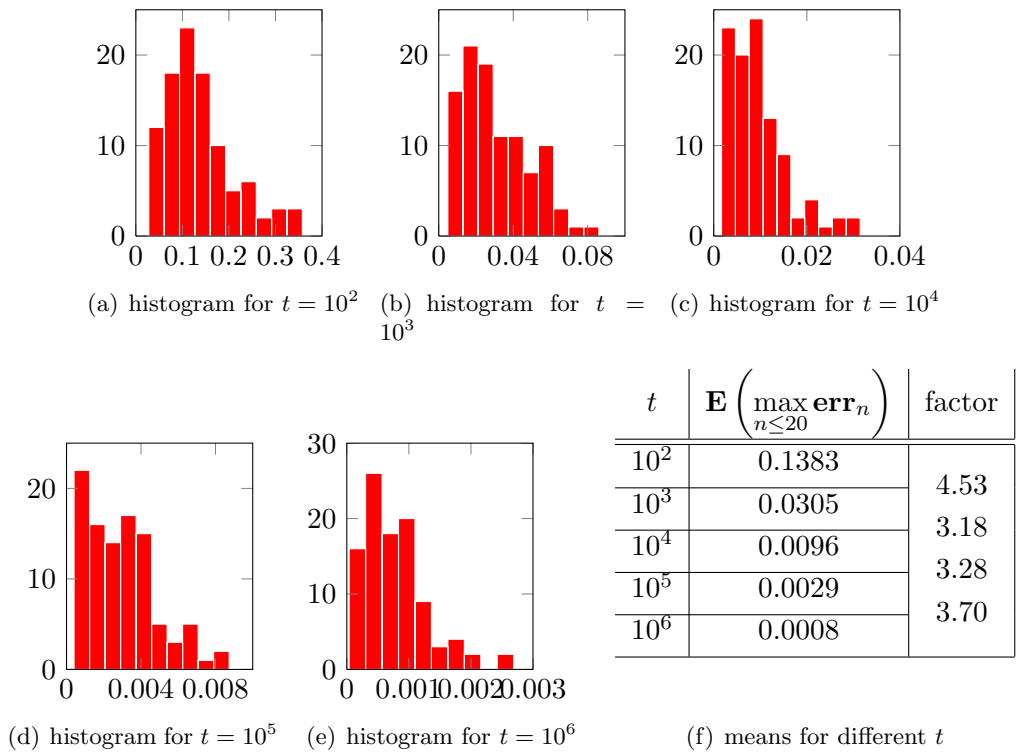
\begin{figure}[!thb]
\setlength\fheight{3cm} \setlength\fwidth{2.45cm}
\centering
\subfigure[histogram for $t= 10^2$]{
%
%
\begin{tikzpicture}

\begin{axis}[%
xtick = {0,0.1,0.2,0.3,0.4},
scale only axis,
width=\fwidth,
height=\fheight,
xmin=0, xmax=0.4,
ymin=0, ymax=25,
axis on top]
\addplot [fill=red,draw=white] coordinates{ (0.027883,0) (0.027883,12) (0.0608618,12) (0.0608618,0) (0.0608618,0) (0.0608618,18) (0.0938406,18) (0.0938406,0) (0.0938406,0) (0.0938406,23) (0.126819,23) (0.126819,0) (0.126819,0) (0.126819,18) (0.159798,18) (0.159798,0) (0.159798,0) (0.159798,10) (0.192777,10) (0.192777,0) (0.192777,0) (0.192777,5) (0.225756,5) (0.225756,0) (0.225756,0) (0.225756,6) (0.258735,6) (0.258735,0) (0.258735,0) (0.258735,2) (0.291714,2) (0.291714,0) (0.291714,0) (0.291714,3) (0.324692,3) (0.324692,0) (0.324692,0) (0.324692,3) (0.357671,3) (0.357671,0) (0.027883,0)};
\end{axis}
\end{tikzpicture}
}
\subfigure[histogram for $t= 10^3$]{
%
%
\begin{tikzpicture}

\begin{axis}[%
xtick = {0,0.04,0.08},
xticklabels = {0,0.04,0.08},
scale only axis,
width=\fwidth,
height=\fheight,
xmin=0, xmax=0.1,
ymin=0, ymax=25,
axis on top]
\addplot [fill=red,draw=white] coordinates{ (0.00490657,0) (0.00490657,16) (0.0130143,16) (0.0130143,0) (0.0130143,0) (0.0130143,21) (0.021122,21) (0.021122,0) (0.021122,0) (0.021122,19) (0.0292296,19) (0.0292296,0) (0.0292296,0) (0.0292296,11) (0.0373373,11) (0.0373373,0) (0.0373373,0) (0.0373373,11) (0.045445,11) (0.045445,0) (0.045445,0) (0.045445,7) (0.0535527,7) (0.0535527,0) (0.0535527,0) (0.0535527,10) (0.0616604,10) (0.0616604,0) (0.0616604,0) (0.0616604,3) (0.0697681,3) (0.0697681,0) (0.0697681,0) (0.0697681,1) (0.0778758,1) (0.0778758,0) (0.0778758,0) (0.0778758,1) (0.0859835,1) (0.0859835,0) (0.00490657,0)};
\end{axis}
\end{tikzpicture}
}
\subfigure[histogram for $t= 10^4$]{
%
%
\begin{tikzpicture}

\begin{axis}[%
xtick = {0,0.02,0.04},
xticklabels = {0,0.02,0.04},
scale only axis,
width=\fwidth,
height=\fheight,
xmin=0, xmax=0.04,
ymin=0, ymax=25,
axis on top]
\addplot [fill=red,draw=white] coordinates{ (0.00168076,0) (0.00168076,23) (0.00465686,23) (0.00465686,0) (0.00465686,0) (0.00465686,20) (0.00763295,20) (0.00763295,0) (0.00763295,0) (0.00763295,24) (0.010609,24) (0.010609,0) (0.010609,0) (0.010609,13) (0.0135851,13) (0.0135851,0) (0.0135851,0) (0.0135851,9) (0.0165612,9) (0.0165612,0) (0.0165612,0) (0.0165612,2) (0.0195373,2) (0.0195373,0) (0.0195373,0) (0.0195373,4) (0.0225134,4) (0.0225134,0) (0.0225134,0) (0.0225134,1) (0.0254895,1) (0.0254895,0) (0.0254895,0) (0.0254895,2) (0.0284656,2) (0.0284656,0) (0.0284656,0) (0.0284656,2) (0.0314417,2) (0.0314417,0) (0.00168076,0)};
\end{axis}
\end{tikzpicture}
}
\\
\subfigure[histogram for $t= 10^5$]{
%
%
\begin{tikzpicture}

\begin{axis}[%
xtick = {0,0.004,0.008},
xticklabels = {0,0.004,0.008},
scale only axis,
width=\fwidth,
height=\fheight,
xmin=0, xmax=0.01,
ymin=0, ymax=25,
axis on top]
\addplot [fill=red,draw=white] coordinates{ (0.000405405,0) (0.000405405,22) (0.00123821,22) (0.00123821,0) (0.00123821,0) (0.00123821,16) (0.00207101,16) (0.00207101,0) (0.00207101,0) (0.00207101,14) (0.00290382,14) (0.00290382,0) (0.00290382,0) (0.00290382,17) (0.00373662,17) (0.00373662,0) (0.00373662,0) (0.00373662,15) (0.00456942,15) (0.00456942,0) (0.00456942,0) (0.00456942,5) (0.00540223,5) (0.00540223,0) (0.00540223,0) (0.00540223,3) (0.00623503,3) (0.00623503,0) (0.00623503,0) (0.00623503,5) (0.00706783,5) (0.00706783,0) (0.00706783,0) (0.00706783,1) (0.00790064,1) (0.00790064,0) (0.00790064,0) (0.00790064,2) (0.00873344,2) (0.00873344,0) (0.000405405,0)};
\end{axis}
\end{tikzpicture}
}
\subfigure[histogram for $t= 10^6$]{
%
%
\begin{tikzpicture}

\begin{axis}[%
xtick = {0,0.001,0.002,0.003},
xticklabels = {0,0.001,0.002,0.003},
scale only axis,
width=\fwidth,
height=\fheight,
xmin=0, xmax=0.003,
ymin=0, ymax=30,
axis on top
]
\addplot [fill=red,draw=white] coordinates{ (3.76262e-05,0) (3.76262e-05,16) (0.000302741,16) (0.000302741,0) (0.000302741,0) (0.000302741,26) (0.000567856,26) (0.000567856,0) (0.000567856,0) (0.000567856,18) (0.000832972,18) (0.000832972,0) (0.000832972,0) (0.000832972,20) (0.00109809,20) (0.00109809,0) (0.00109809,0) (0.00109809,9) (0.0013632,9) (0.0013632,0) (0.0013632,0) (0.0013632,3) (0.00162832,3) (0.00162832,0) (0.00162832,0) (0.00162832,4) (0.00189343,4) (0.00189343,0) (0.00189343,0) (0.00189343,2) (0.00215855,2) (0.00215855,0) (0.00215855,0) (0.00215855,0) (0.00242366,0) (0.00242366,0) (0.00242366,0) (0.00242366,2) (0.00268878,2) (0.00268878,0) (3.76262e-05,0)};
\end{axis}
\end{tikzpicture}
}
\subfigure[means for different $t$]{
\raisebox{15ex}{\begin{tabular}{c|c|c|}
$t$ & $\EW\left(\max\limits_{n \leq 20} \err_n\right)$ & factor\\[1ex]
\hline\hline
\rule{0cm}{.3cm}$10^2$ & 0.1383 & \\[.5ex]
\cline{1-2}
\rule{0cm}{.3cm}$10^3$ & 0.0305 &\raisebox{1.5ex}[1.5ex]{$4.53$}\\[.5ex]
\cline{1-2}
\rule{0cm}{.3cm}$10^4$ & 0.0096 &\raisebox{1.5ex}[1.5ex]{$3.18$}\\[.5ex]
\cline{1-2}
\rule{0cm}{.3cm}$10^5$ & 0.0029 &\raisebox{1.5ex}[1.5ex]{$3.28$}\\[.5ex]
\cline{1-2}
\rule{0cm}{.3cm}$10^6$ & 0.0008 &\raisebox{1.5ex}[1.5ex]{$3.70$}\\[.5ex]
\hline
\end{tabular}}
}
\caption{Overview for the error terms \eqref{eq:defi_errnB} for the inverse scattering problem. For different values of the expected total number of counts the value $\max_{n \leq 20} \err_n$ has been calculated in 100 experiments. The
figure shows the corresponding histograms and means. The decay of order $\frac{1}{\sqrt{t}}$, i.e.\ reduction by a factor of $\sqrt{10} \approx 3.16$ in the table is clearly visible. 
All parameters are as in Figure \ref{fig:scattering}.}
\end{figure}

For comparison we take the usual IRGNM, i.e.\ \eqref{eqs:method} with $\S{g}{\hat g} =\left\Vert g-\hat g\right\Vert_{L^2}^2$ and $\mathcal R$ as above as well as a weighted IRGNM where $\calS$ is chosen to be Pearson's $\phi^2$-distance:
\begin{equation}\label{eq:pearson}
\phi^2 \left(\gdelta;g\right) = \int\limits_{\manifold} \frac{\left|g- \gdelta\right|^2}{\gdelta}\,\mathrm d x.
\end{equation}
Since in all our examples we have many zero counts, we actually used
\[
\S{g}{\gdelta} = \phi^2\left(\gdelta;\max\{g,c\}\right)
\]
with a cutoff-parameter $c>0$. 

Error statistics of shape reconstructions from 100 experiments are shown in Table \ref{tab:scattering}. The stopping index
$N$ is chosen a priori such that (the empirical version of) the expectation 
$\EW \|q_n-q^{\dagger}\|_{L^2}^2$ is minimal for $n=N$, i.e.\ we compare both 
methods with an oracle stopping rule. Note that the mean square error is significantly smaller for the Kullback-Leibler divergence than
for the $L^2$-distance and also clearly smaller than for Pearson's distance.
Moreover the distribution of the error is more concentrated for the
Kullback-Leibler divergence. 
For Pearson's $\phi^2$ distance it must be said that the results depend strongly on the cutoff parameter for the data. In our experiments $c=0.2$ 
seemed to be a good choice in general.

\medskip
\paragraph{\bf A phase retrieval problem.} 

\begin{figure}[!htb]
\setlength\fheight{4cm} \setlength\fwidth{4cm} 
\centering
\subfigure[exact solution]{
\label{subfig:phase_retrieval_exact}
%
%
\begin{tikzpicture}

\begin{axis}[%
view={0}{90},
name=plot1,
scale only axis,
width=\fwidth,
height=\fheight,
y dir=reverse,
xmin=-0.5, xmax=255.5,
ymin=-0.5, ymax=255.5,
axis on top]
\addplot graphics [xmin=-5.000000e-01, xmax=2.555000e+02, ymin=-5.000000e-01, ymax=2.555000e+02] {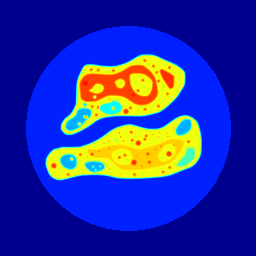};
\end{axis}

\begin{axis}[%
axis on top,
at=(plot1.right of south east), anchor=left of south west,
width=0.0675676\fwidth, height=1\fheight,
scale only axis,
xmin=0, xmax=1,
ymin=0, ymax=2.2,
xtick=\empty, yticklabel pos=right]
\addplot graphics [xmin=0, xmax=1, ymin=0, ymax=2.200000e+00] {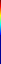};
\end{axis}
\end{tikzpicture} 
}
\subfigure[$\log_{10}$ of observed count data $\gdelta$]{
\label{subfig:phase_retrieval_data}
%
%
\begin{tikzpicture}

\begin{axis}[%
view={0}{90},
name=plot1,
scale only axis,
width=\fwidth,
height=\fheight,
y dir=reverse,
xmin=0.5, xmax=256.5,
ymin=0.5, ymax=256.5,
axis on top]
\addplot graphics [xmin=5.000000e-01, xmax=2.565000e+02, ymin=5.000000e-01, ymax=2.565000e+02] {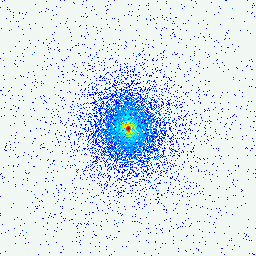};
\end{axis}

\begin{axis}[%
axis on top,
at=(plot1.right of south east), anchor=left of south west,
width=0.0675676\fwidth, height=1\fheight,
scale only axis,
xmin=0, xmax=1,
ymin=-3, ymax=5.796,
xtick=\empty, yticklabel pos=right]
\addplot graphics [xmin=0, xmax=1, ymin=-3, ymax=5.796000e+00] {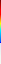};
\end{axis}
\end{tikzpicture}
}\\
\subfigure[median reconstruction for the IRGNM]{
\label{subfig:phase_retrieval_l2}
%
%
\begin{tikzpicture}

\begin{axis}[%
view={0}{90},
name=plot1,
scale only axis,
width=\fwidth,
height=\fheight,
y dir=reverse,
xmin=-0.5, xmax=255.5,
ymin=-0.5, ymax=255.5,
axis on top]
\addplot graphics [xmin=-5.000000e-01, xmax=2.555000e+02, ymin=-5.000000e-01, ymax=2.555000e+02] {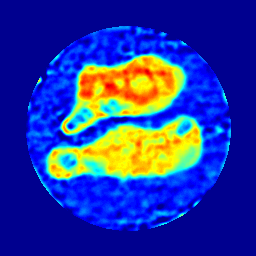};
\end{axis}

\begin{axis}[%
axis on top,
at=(plot1.right of south east), anchor=left of south west,
width=0.0675676\fwidth, height=1\fheight,
scale only axis,
xmin=0, xmax=1,
ymin=0, ymax=2.2,
xtick=\empty, yticklabel pos=right]
\addplot graphics [xmin=0, xmax=1, ymin=0, ymax=2.200000e+00] {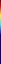};
\end{axis}
\end{tikzpicture} 
}
\subfigure[$\log_{10}$ of exact data $tF(\varphi^{\dagger})$]{
\label{subfig:phase_retrieval_data_exact}
%
%
\begin{tikzpicture}

\begin{axis}[%
view={0}{90},
name=plot1,
scale only axis,
width=\fwidth,
height=\fheight,
y dir=reverse,
xmin=0.5, xmax=256.5,
ymin=0.5, ymax=256.5,
axis on top]
\addplot graphics [xmin=5.000000e-01, xmax=2.565000e+02, ymin=5.000000e-01, ymax=2.565000e+02] {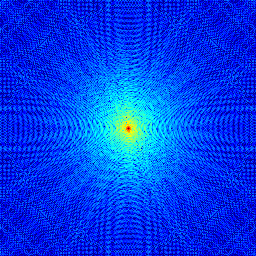};
\end{axis}

\begin{axis}[%
axis on top,
at=(plot1.right of south east), anchor=left of south west,
width=0.0675676\fwidth, height=1\fheight,
scale only axis,
xmin=0, xmax=1,
ymin=-3, ymax=5.79626,
xtick=\empty, yticklabel pos=right]
\addplot graphics [xmin=0, xmax=1, ymin=-3, ymax=5.796263e+00] {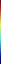};
\end{axis}
\end{tikzpicture} 
}\\
\subfigure[median reconstruction for our method \eqref{eq:method_impl}]{
\label{subfig:phase_retrieval_kl}
%
%
\begin{tikzpicture}

\begin{axis}[%
view={0}{90},
name=plot1,
scale only axis,
width=\fwidth,
height=\fheight,
y dir=reverse,
xmin=-0.5, xmax=255.5,
ymin=-0.5, ymax=255.5,
axis on top]
\addplot graphics [xmin=-5.000000e-01, xmax=2.555000e+02, ymin=-5.000000e-01, ymax=2.555000e+02] {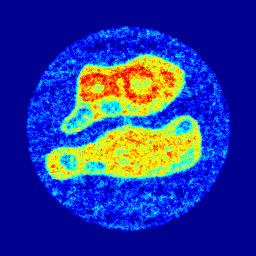};
\end{axis}

\begin{axis}[%
axis on top,
at=(plot1.right of south east), anchor=left of south west,
width=0.0675676\fwidth, height=1\fheight,
scale only axis,
xmin=0, xmax=1,
ymin=0, ymax=2.2,
xtick=\empty, yticklabel pos=right]
\addplot graphics [xmin=0, xmax=1, ymin=0, ymax=2.200000e+00] {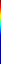};
\end{axis}
\end{tikzpicture} 
}
\subfigure[$\log_{10}$ of median data reconstruction $tF(\varphi_N)$
for our method \eqref{eq:method_impl}]{
\label{subfig:phase_retrieval_data_rec}
%
%
\begin{tikzpicture}

\begin{axis}[%
view={0}{90},
name=plot1,
scale only axis,
width=\fwidth,
height=\fheight,
y dir=reverse,
xmin=0.5, xmax=256.5,
ymin=0.5, ymax=256.5,
axis on top]
\addplot graphics [xmin=5.000000e-01, xmax=2.565000e+02, ymin=5.000000e-01, ymax=2.565000e+02] {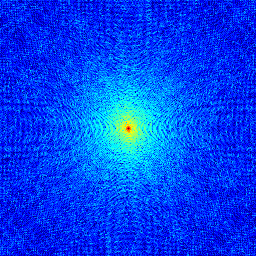};
\end{axis}

\begin{axis}[%
axis on top,
at=(plot1.right of south east), anchor=left of south west,
width=0.0675676\fwidth, height=1\fheight,
scale only axis,
xmin=0, xmax=1,
ymin=-3, ymax=5.79667,
xtick=\empty, yticklabel pos=right]
\addplot graphics [xmin=0, xmax=1, ymin=-3, ymax=5.796673e+00] {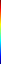};
\end{axis}
\end{tikzpicture} 
}
\caption{Median reconstructions for the phase retrieval problem with 
$t=10^6$ expected counts.}
\label{fig:phaseretrieval}
\end{figure}

A well-known class of inverse problems with numerous applications in optics
consists in reconstructing a function $f:\mathbb{R}^d\to\mathbb{C}$ from the modulus of 
its Fourier transform $|\mathcal{F} f|$
and additional a priori information, or equivalently to reconstruct the phase
$\mathcal{F}f/|\mathcal{F}f|$ of $\mathcal{F}f$ (see Hurt \cite{h89}). 

In the following we assume more specifically that $f:\mathbb{R}^2\to \mathbb{C}$ is of 
the form $f(x)= \exp(\textup{i} \varphi(x))$ with an unknown real-valued function 
$\varphi$ with known compact support $\mathrm{supp}(\varphi)$. 
For a uniqueness result we refer to Klibanov \cite{klibanov:06}, although
not all assumptions of this theorem are satisfied in the example below. 
It turns out to be particularly
helpful if $\varphi$ has a jump of known magnitude at the boundary of its support.
We will assume that $\mathrm{supp}\,  \varphi = B_{\rho} = \{x\in\mathbb{R}^2:|x|\leq \rho\}$
and that $\varphi \approx \chi_{B_{\rho}}$ close to the boundary $\partial B_{\rho}$ (here
$\chi_{B_{\rho}}$ denotes the characteristic function of $B_{\rho}$). This leads to an inverse
problem where the forward operator is given by 
\begin{align}\label{eq:phaseretrievalop}
\begin{aligned}
F&: H^s(B_{\rho})  \longrightarrow L^\infty(\manifold)\,, \\
& (F\varphi)(\xi) := \left|\int_{B_\rho} e^{-\textup{i}\xi\cdot x} e^{\textup{i}\varphi(x)}\,\mathrm dx\right|^2\,.
\end{aligned}
\end{align}
Here $H^s(B_\rho)$ denotes a Sobolev space with index $s\geq 0$ and 
$\manifold\subset\mathbb{R}^2$ is typically of the form $\manifold = [-\kappa,\kappa]^2$. 
The a priori information on $\varphi$ can be incorporated in the form of an initial guess
$\varphi_0\equiv 1$. Note that the range of $F$ consists of analytic functions.

The problem above occurs in optical imaging: If $f(x')=\exp(\textup{i}\varphi(x')) =u(x',0)$ ($x'=(x_1,x_2)$)
denotes the values of a cartesian component $u$ of an electric field in
the plane $\{x\in\mathbb{R}^3:x_3=0\}$ and $u$ solves the Helmholtz equation
$\Delta u +k^2u=0$ 
and a radiation condition in the half-space $\{x\in\mathbb{R}^3:x_3>0\}$, 
then the intensity $g(x') = |u(x',\Delta)|^2$ of the electric field at a 
measurement plane $\{x\in\mathbb{R}^3: x_3=\Delta\}$ in the limit $\Delta\to \infty$ in the
\emph{Fraunhofer approximation} is given by $|\mathcal{F}_2f|^2$ up to rescaling
(see e.g.\ 
Paganin \cite[Sec. 1.5]{paganin:06}).
If $f$ is generated by a plane incident wave in $x_3$ direction passing through a non-absorbing, weakly scattering object of interest 
in the half-space $\{x_3<0\}$ close to the plane $\{x_3=0\}$ and if the wave
length is small compared to the length scale of the object, 
then the \emph{projection approximation} 
$\varphi(x')\approx \frac{k}{2}\int_{-\infty}^0 (n^2(x',x_3)-1)\,\mathrm dx_3$ is valid 
where  $n$ describes the refractive index of the object of interest 
(see e.g.\ \cite[Sec. 2.1]{paganin:06}). 
A priori information on $\varphi$ concerning a jump 
at the boundary of its support can be obtained by placing a known transparent object
before or behind the object or interest. 

The simulated test object in Figure \ref{fig:phaseretrieval} which represents 
two cells is taken from Giewekemeyer et al.\ \cite{giewekemeyer_etal:11}. 
We choose the initial guess $\varphi_0\equiv 1$, the Sobolev index $s=\frac12$, 
and the regularization parameters $\alpha_n=\frac{5}{10^6}\cdot(2/3)^n$. 
The photon density is approximated by $J=256^2$ bins. 
The offset parameter $\offset$ is initially set to $2\cdot 10^{-6}$ and reduced by 
a factor $\frac45$ in each iteration step. As for the scattering problem, 
we use an oracle stopping rule 
$N := \mathrm{argmin}_n\EW \|\varphi_n-\varphi^{\dagger}\|_{L^2}^2$.  
As already mentioned, we had difficulties to solve the quadratic minimization
problems \eqref{eq:method_impl} by the CG method for small $\alpha_n$ and had to stop
the iterations before residuals were sufficiently small to guarantee a 
reliable solution. 

Nevertheless, comparing subplots (c) and (e)
in Figure \ref{fig:phaseretrieval}, the median KL-reconstruction (e) seems preferable  
(although more noisy) since the contours are sharper and details in the interior of the 
cells are more clearly separated.

\section*{Acknowledgement}
We would like to thank Tim Salditt and Klaus Giewekemeyer for helpful 
discussions and data concerning the phase retrieval problem, 
Patricia Reynaud-Bouret for fruitful discussions on concentration
inequalities, and two anonymous referees for their suggestions,
which helped to improve the paper considerably. 
Financial support by the German Research Foundation DFG through 
SFB 755, the Research Training Group 1023 and the Federal Ministry of 
Education and Research (BMBF) through the project INVERS is gratefully 
acknowledged.

\small
\bibliography{bib}{}
\bibliographystyle{plain}
\end{document}